\documentclass[11pt]{article}
\usepackage{amssymb}
\usepackage{amsmath}
\usepackage{mathrsfs}
\usepackage{graphics}
\usepackage{graphicx}
\usepackage{xcolor}
\usepackage{subfigure}
\usepackage[T1]{fontenc}
\usepackage{latexsym,amssymb,amsmath,amsfonts,amsthm}\usepackage{txfonts}
\newcommand{\be}{\begin{eqnarray}}
\newcommand{\ben}{\begin{eqnarray*}}
\newcommand{\en}{\end{eqnarray}}
\newcommand{\enn}{\end{eqnarray*}}

\newtheorem{theorem}{Theorem}[section]
\newtheorem{lemma}{Lemma}[section]
\newtheorem{prp}[theorem]{Proposition}
\newtheorem{thm}[theorem]{Theorem}

\newtheorem{dfn}{Definition}[section]

\newtheorem{remark}{Remark}

\begin{document}
\renewcommand{\theequation}{\arabic{section}.\arabic{equation}}
\begin{titlepage}
\title{\bf On the small time asymptotics of 3D stochastic primitive equations
}
\author{Zhao Dong$^{1,2}$  Rangrang Zhang$^{3,}$\thanks{Corresponding author.}\\
{\small $^1$ RCSDS, Academy of Mathematics and Systems Science, Chinese Academy of Sciences, Beijing 100190, China.}\\
{\small $^2$  School of Mathematical Sciences, University of Chinese Academy of Sciences.}\\
{\small $^3$ Department of  Mathematics,
Beijing Institute of Technology, Beijing, 100081, China.}\\
({\sf dzhao@amt.ac.cn}, {\sf rrzhang@amss.ac.cn} )}
\date{}
\end{titlepage}
\maketitle

\noindent\textbf{Abstract}:\ \ In this paper, we establish a small time large deviation principle for the strong solution of 3D stochastic primitive equations driven by multiplicative noise, which not only involves the study of small noise, but also the challenging nonlinear drift terms.

%

\section{Introduction}
In this paper, we are concerned with the small time asymptotics of the primitive equations, which are a basic model in the study of large scale oceanic and atmospheric dynamics. This model forms the analytical core of the most general circulation models, which has been intensively investigated because of its challenging nonlinear terms and anisotropic structure (see \cite{L-T-W-1,L-T-W-2,JP} and references therein). The mathematical study of the primitive equations started in a series of pioneer articles by Lions et al. in the early 1990s (see \cite{L-T-W-1,L-T-W-2,L-T-W-3,L-T-W-4}), where
they defined the notions of weak and strong solutions and also showed the global existence of weak solutions. Taking the advantage of the fact that the pressure is essentially two dimensional, Cao and Titi obtained the global well-posedness for the primitive equations in three dimensional case in \cite{C-T-1}. Lately, Hu et al.
\cite{H-T-Z} proved the global existence and uniqueness of strong solutions to the primitive equations under the small depth hypothesis.

\
Due to the existence of some uncertainties, it is natural and reasonable to consider the primitive equations in random case (see \cite{F-H,Guo,M-M,P} and the references therein). In the past two decades, there are numerous works about the stochastic primitive equations.  Guo and Huang \cite{Guo} obtained
the existence of universal random attractor of strong solution to 3D stochastic primitive equations under the assumptions that the momentum
equation is driven by an additive stochastic forcing and the
thermodynamical equation is under a fixed heat source. Debussche et al. \cite{D-G-T-Z} established the global well-posedness of strong solution of the 3D primitive
equations driven by a nonlinear, multiplicative white noise. However, the uniqueness of weak solutions of the 3D primitive equations is open. Based on \cite{D-G-T-Z}, Dong et al. \cite{RR} showed that all special weak solutions obtained by Galerkin approximations shared the same invariant measure. Moreover, Dong and Zhang \cite{RR-1} made Markov selection of weak solutions of the 3D primitive equations and further established that the selected Markov solutions have $\mathcal{W}-$strong Feller property. The Freidlin-Wentzell's large deviation principles (LDP) for the primitive equations also have attracted a lot of people's attention. We mention some of them. Gao and Sun \cite{G-S}  proved that the LDP holds for weak solutions of the 2D stochastic primitive equations. Dong et al. \cite{RR-LDP} obtained the same result for the strong solution of 3D stochastic primitive equations.

The purpose of this paper is to study the small time asymptotics for the strong solution (in the sense of probability) of 3D stochastic primitive equations, which describes the behavior of the fluid at very small time. Specifically, we focus on the limiting behavior of the solution in time interval $[0,t]$ as $t$ goes to zero. An important motivation to consider such problem comes from Varadhan identity
\begin{eqnarray*}
\lim_{t\rightarrow 0}2t\log P(Y(0)\in B,\ Y(t)\in C)=-d^2(B,C),
\end{eqnarray*}
where $d$ is an appropriate Riemann distance associated with the diffusion generated by $Y$.
 The mathematical study of the small time asymptotics for finite dimensional processes was initiated by Varadhan \cite{V}. For the infinite dimensional diffusion processes, the readers can refer to \cite{A-K, A-Z, F-Z, H-R, ZTS} and references therein.

Up to now, there are several works about the small time asymptotics for fluid dynamical models. For example, Xu and Zhang \cite{X-Z} established the small time asymptotics of 2D Navier-Stokes equations in the state space $C([0,\Upsilon];H)$. Liu et al. \cite{Liu} studied 2D quasi-geostrophic equations in the subcritical case, where they obtained the small time asymptotics in the state space $L^{\infty}([0,\Upsilon];H)$ and $L^{\infty}([0,\Upsilon];H^{-\frac{1}{2}})$ for regular initial value in $H^{\delta}\cap L^p$ and $L^p$, respectively. In this paper, we devote to proving the small time asymptotics of the 3D stochastic primitive equations. Apart from the above motivations, its challenging nonlinear drift terms also stimulate our interest. For the small time asymptotics of the strong solution of the 3D stochastic primitive equations, the state space is chosen to be $C([0,\Upsilon];V)$, which requires some higher Sobolev norm estimates. For instance, a key step during the proof process is to show that the probability of the strong solution staying outside an energy ball in $V\subset H^1$ is exponentially small. It's difficult to achieve this directly like 2D Navier-Stokes equations or 2D quasi-geostrophic equations, since the nonlinear terms of 3D primitive equations have no cancellation property in $V$ and $L^p-$norm are not strong enough to control them  (see Theorem \ref{thm-2}). Fortunately, we can overcome this difficulty by introducing appropriate stopping times. Under this circumstance, we have to make additional estimates on those stopping times. These are highly nontrivial. All details are presented in Sect. \ref{Sect3.2}.

\

This paper is organized as follows. The mathematical framework of 3D stochastic primitive equations is introduced in Sect. 2. In Sect. 3, the main result of this paper is stated and the proof of small time asymptotics of 3D stochastic primitive equations is given.

\section{The mathematical framework}
Let $D$ be a bounded open domain with smooth boundary in $\mathbb{R}^2$. Set $\mathcal{O}=D\times (-1,0)$. Let $\Upsilon>0$ and $[0, \Upsilon]$ be an time interval. Consider the following
3D primitive equations on $\mathcal{O}\times [0,\Upsilon]$ driven by a stochastic forcing in a Cartesian system
\begin{eqnarray}\label{eq-1}
\frac{\partial \mathbf{v}}{\partial t}+(\mathbf{v}\cdot \nabla)\mathbf{v}+\theta\frac{\partial \mathbf{v}}{\partial z}+f{k}\times \mathbf{v} +\nabla P +L_1\mathbf{v} &=&\psi_1(t,\mathbf{v},T)\frac{dW_1}{dt},\\
\label{eq-2}
\partial_{z}P+T&=&0,\\
\label{eq-3}
\nabla\cdot \mathbf{v}+\partial_{z}\theta&=&0, \\
\label{eq-4}
\frac{\partial T}{\partial t}+(\mathbf{v}\cdot\nabla)T+\theta\frac{\partial T}{\partial z}+L_2 T&=&\psi_2(t,\mathbf{v},T)\frac{dW_2}{dt},
\end{eqnarray}
where the horizontal velocity field $\mathbf{v}=(v_{1},v_{2})$, the vertical velocity field\ $\theta$, the temperature\ $T$ and the pressure\ $P$ are all unknown functionals. $f$ is the Coriolis parameter. ${k}$ is vertical unit vector. $W_1$ and $W_2$ are two independent cylindrical Winner processes on $U_1$ and $U_2$, respectively. $U_1$ and $U_2$ will be given in Sect. \ref{sect}.
$\nabla=(\partial x,\partial y)$, $\Delta=\partial^{2}_{x}+\partial^{2}_{y}$. The viscosity and the heat diffusion operators $L_1$ and $L_2$ are given by
\begin{eqnarray*}
L_1\mathbf{v}&=&-A_1\Delta \mathbf{v} -A_2\frac{\partial^2 \mathbf{v}}{\partial z^2},\\
L_2T&=&-K_1\Delta T -K_2\frac{\partial^2 T}{\partial z^2},
\end{eqnarray*}
where $A_1$, $A_2$ are positive molecular viscosities and $K_1$, $K_2$ are positive conductivity constants. Without loss of generality, we assume that
$$
A_1=A_2=K_1=K_2=1.
$$
We impose the same boundary conditions as \cite{D-G-T-Z},
\begin{eqnarray}\label{eq1}
 \partial_{z}\mathbf{v}=0, \ \theta=0,\  \partial_{z}T=0  & &  {\rm {on}} \ \Gamma_{\textit{u}}:=\textit{D}\times\{0\},\\
 \label{eq2}
 \partial_{z}\mathbf{v}=0,\  \theta=0,\  \partial_{z}T=0        & &   {\rm {on}}\ \Gamma_{\textit{b}}:=\textit{D}\times\{-1\}\\
 \label{eq3}
 \mathbf{v}=0,\ \frac{\partial T}{\partial \mathbf{n}}=0  & &   {\rm {on}}\   \Gamma_{l}:=\partial \textit{D}\times [-1,0],
 \end{eqnarray}
 where $\mathbf{n}$ is the outward normal vector to $\Gamma_{l}$.

Integrating (\ref{eq-3}) from $-1$ to $z$ and using  (\ref{eq1}), (\ref{eq2}), we have
\begin{equation}
\theta(t,x,y,z):=\Phi(\mathbf{v} )(t,x,y,z)=-\int^{z}_{-1}\nabla\cdot \mathbf{v} (t,x,y,z')dz',
\end{equation}
moreover,
\[
\int^{0}_{-1}\nabla\cdot {\mathbf{v}}  dz=0.
\]
Integrating (\ref{eq-2}) from $-1$ to $z$, set $p_{b}$ be a certain unknown function at $\Gamma_{b}$ satisfying
\begin{eqnarray}\label{e-51}
P(x,y,z,t)= p_{b}(x,y,t)-\int^{z}_{-1} T(x,y,z',t) dz'.
\end{eqnarray}
Then, (\ref{eq-1})-(\ref{eq3}) can be rewritten as
 \begin{eqnarray}\label{eq5-1}
&\frac{\partial \mathbf{v}}{\partial t}+(\mathbf{v}\cdot \nabla)\mathbf{v}+\Phi(\mathbf{v})\frac{\partial \mathbf{v}}{\partial z}+f{k}\times \mathbf{v} +\nabla p_{b}-\int^{z}_{-1}\nabla T dz' + L_1\mathbf{v} =\psi_1(t,\mathbf{v},T)\frac{dW_1}{dt},&\\
\label{eq-6-1}
&\frac{\partial T}{\partial t}+(\mathbf{v}\cdot\nabla)T+\Phi(\mathbf{v})\frac{\partial T}{\partial z}+ L_2T=\psi_2(t,\mathbf{v},T)\frac{dW_2}{dt},&\\
\label{eq-7-1}
&\int^{0}_{-1}\nabla\cdot \mathbf{v}  dz=0.&
\end{eqnarray}
\begin{eqnarray}\label{eq-8-1}
 \partial_{z}\mathbf{v}=0,\ \partial_{z}T=0    &&      {\rm on}\ \Gamma_{\textit{u}},\\
 \label{eq-9-1}
 \partial_{z}\mathbf{v}=0,\ \partial_{z}T=0     &&   {\rm on}\ \Gamma_{\textit{b}},\\
 \label{eq-10-1}
 \mathbf{v}=0,\ \frac{\partial T}{\partial \mathbf{n}}=0      &&   {\rm on}\ \Gamma_{l}.
 \end{eqnarray}
Denote $Y =({\mathbf{v}}, T)$ and the initial value conditions
 \begin{equation}\label{eq-11-1}
 Y(0)=({\mathbf{v}}(0), T(0))=\gamma:=(\gamma_1,\gamma_2).
 \end{equation}
\subsection{Some functional spaces}\label{sect}
Let $\mathcal{L}(K_1;K_2)$ (resp. $\mathcal{L}_2(K_1;K_2)$) be the space of bounded (resp. Hilbert-Schmidt) linear operators from the Hilbert space $K_1$ to $K_2$, whose norm is denoted by $\|\cdot\|_{\mathcal{L}(K_1;K_2)}$(resp. $\|\cdot\|_{\mathcal{L}_2(K_1;K_2)})$. For $p\in \mathbb{Z}^+$, set
\begin{eqnarray*}
 |\phi|_p=\left\{
            \begin{array}{ll}
              \Big(\int_{\mathcal{O}}|\phi(x,y,z)|^pdxdydz\Big)^{\frac{1}{p}}, & \phi\in L^p(\mathcal{O}), \\
               \Big(\int_{D}|\phi(x,y)|^pdxdy\Big)^{\frac{1}{p}}, &  \phi\in L^p(D).
            \end{array}
          \right.
\end{eqnarray*}
In particular, $|\cdot|$ and $(\cdot,\cdot)$ represent norm and inner product of $L^2(\mathcal{O})$ (or $L^2(D)$), respectively. For the classical Sobolev space $H^{m}(\mathcal{O})$, $m\in \mathbb{N}_+$,
\begin{equation}\notag
\left\{
  \begin{array}{ll}
    H^{m}(\mathcal{O})=\Big\{Y\Big| \partial_{\alpha}Y\in (L^2(\mathcal{O}))^3\ {\rm for} \ |\alpha|\leq m\Big\},&  \\
    |Y|^2_{H^{m}(\mathcal{O})}=\sum_{0\leq|\alpha|\leq m}|\partial_{\alpha}U|^2. &
  \end{array}
\right.
\end{equation}
It's known that $(H^{m}(\mathcal{O}), |\cdot|_{H^{m}(\mathcal{O})})$ is a Hilbert space. $|\cdot|_{H^p(D)}$ stands for the norm of $H^p(D)$ for $p\in \mathbb{Z}^+$.

In the following, we adopt the framework of (\ref{eq5-1})-(\ref{eq-11-1}) in \cite{D-G-T-Z}. Define
 \begin{eqnarray}
H:=\Big\{({\mathbf{v}},T)\in (L^2(\mathcal{O}))^3:\nabla\cdot \int^0_{-1}\mathbf{v}dz'=0 \ {\rm in}\ D,\ \mathbf{n} \cdot \int^0_{-1}\mathbf{v}dz'=0 \ {\rm on}\  \partial D, \ \int_{\mathcal{O}}T dxdydz=0\Big\}.
 \end{eqnarray}
$H$ is equipped with $L^2(\mathcal{O})$ inner product. Define $P_H$ to be the Leray type projection operator from $(L^2(\mathcal{O}))^3$ onto $H$. Consider the subspace of $(H^1(\mathcal{O}))^3$,
 \begin{eqnarray}
V:=\Big\{(\mathbf{v},T)\in (H^1(\mathcal{O}))^3:\nabla\cdot \int^0_{-1}\mathbf{v}dz'=0 \ {\rm in}\ D,\ \mathbf{v}=0 \ {\rm on}\  \Gamma_l , \ \int_{\mathcal{O}}T dxdydz=0\Big\}.
 \end{eqnarray}
$V$ is equipped with the inner product
 \begin{eqnarray*}
 ((Y,\tilde{Y}))&:=&((\mathbf{v},\tilde{\mathbf{v}}))_1+((T,\tilde{T}))_2,\\
 ((\mathbf{v},\tilde{\mathbf{v}}))_1&:=&\int_{\mathcal{O}}(\nabla \mathbf{v}\cdot \nabla \tilde{\mathbf{v}}+\partial_z \mathbf{v}\cdot \partial_z \tilde{\mathbf{v}})dxdydz,\\
 ((T,\tilde{T}))_2&:=&\int_{\mathcal{O}}(\nabla T\cdot \nabla \tilde{T}+\partial_z T \partial_z \tilde{T})dxdydz,
 \end{eqnarray*}
 and take $\|\cdot\|=\sqrt{((\cdot, \cdot))}$. For the individual components of the solution $Y=({\mathbf{v}},T)$, for convenience, $|\cdot|$ and $\|\cdot\|$ are also used for ${\mathbf{v}}=(v_1, v_2)$ and $T$. Moreover, for $p\geq 1$, let
\[
|{\mathbf{v}}|_{p}:=\left(\int_{\mathcal{O}}(|v_1|^p+|v_2|^p)dxdydz\right)^{\frac{1}{p}}.
\]
The principle linear portion of the equation is defined by
\begin{eqnarray*}
AY:=P_H \left(                 
  \begin{array}{c}   
   L_1 {\mathbf{v}}\\  
    L_2 T \\  
  \end{array}
\right) \quad {\rm{for}} \ Y=({\mathbf{v}},T)\in D(A)
\end{eqnarray*}
where
\begin{eqnarray*}
D(A)&=&\Big\{Y=({\mathbf{v}},T)\in (H^2(\mathcal{O}))^3: \partial_z {\mathbf{v}}=\partial_z T=0 \ \ {\rm{on}} \ \Gamma_u,\ \frac{\partial T}{\partial \mathbf{n}}=0 \ \ {\rm{on}}\ \Gamma_l,\\
&& \ \partial_z {\mathbf{v}}=\partial_z T=0 \ \ {\rm{on}} \ \Gamma_b\Big\}.
\end{eqnarray*}
It is well-known that $A$ is a self-adjoint and positive definite operator. Due to the regularity properties of the Stokes problem of geophysical fluid dynamics, we have $|AY|\cong |Y|_{H^2(\mathcal{O})}$, see \cite{Z}.\\
For $Y=({\mathbf{v}},T)$, $\tilde{Y}=(\tilde{{\mathbf{v}}},\tilde{T})\in D(A)$, define
\begin{eqnarray*}
B(Y, \tilde{Y}):=P_H \left(                 
  \begin{array}{c}   
   (\mathbf{v}\cdot \nabla) \tilde{\mathbf{v}}+\Phi(\mathbf{v})\frac{\partial \tilde{\mathbf{v}}}{\partial z}\\  
    (\mathbf{v}\cdot \nabla) \tilde{T}+\Phi(\mathbf{v})\frac{\partial \tilde{T}}{\partial z} \\  
  \end{array}
\right).
\end{eqnarray*}
For the Coriolis forcing term and the second component of pressure in (\ref{e-51}), define
\begin{eqnarray*}
G(Y):=P_H \left(                 
  \begin{array}{c}   
   fk\times {\mathbf{v}}-\int^z_{-1}\nabla T dz'\\  
    0\\  
  \end{array}
\right) \quad {\rm{for}} \ Y=({\mathbf{v}},T)\in V.
\end{eqnarray*}
For stochastic terms, we shall fix a single stochastic basis $\mathcal{T}:=(\Omega, \mathcal{F}, \{\mathcal{F}_t\}_{t\geq 0}, P, W)$.
\[
W=\left(                 
  \begin{array}{c}   
    W_1\\  
    W_2 \\  
  \end{array}
\right)
 \]
 is a cylindrical Brownian motion with the form $W(t,\omega)=\sum_{i\geq1}r_iw_i(t,\omega)$, where $\{r_i\}_{i\geq 1}$ is a complete orthonormal basis of a Hilbert space $U=U_1\times U_2$, $U_1$ and $U_2$ are separable Hilbert spaces, $\{w_i\}_{i\geq1}$ is a sequence of independent one-dimensional standard Brownian motions on $(\Omega, \mathcal{F}, \{\mathcal{F}_t\}_{t\geq 0}, P)$.
Define
\begin{equation}\notag
\psi(t,Y):=P_H\left(                 
  \begin{array}{cc}   
   \psi_1(t,Y) & 0 \\  
   0 & \psi_2(t,Y) \\  
  \end{array}
\right).
\end{equation}
Using the above operators, we reformulate (\ref{eq5-1})-(\ref{eq-11-1}) as the following abstract evolution system
\begin{eqnarray}\label{equ-7}
\left\{
  \begin{array}{ll}
    dY(t)+AY(t)dt+B(Y(t),Y(t))dt+G(Y(t))dt=\psi(t,Y(t)) dW(t), \\
    Y(0)=\gamma.
  \end{array}
\right.
\end{eqnarray}
\subsection{Hypothesis}
Recall the definition of strong solution to (\ref{equ-7}) stated in \cite{D-G-T-Z}.
\begin{dfn}\label{dfn-3}
Let $\mathcal{T}=(\Omega, \mathcal{F}, \{\mathcal{F}_t\}_{t\geq 0}, P, W)$ be a fixed stochastic basis and suppose that $\gamma$ is a $V-$valued $\mathcal{F}_0-$measurable random variable with $E\|\gamma\|^2<\infty$.
$Y$ is called a strong solution of (\ref{equ-7}) if $Y(\cdot)$ is an $\mathcal{F}_t-$adapted process in $V$,  such that
\[
Y(\cdot)\in L^2(\Omega; C([0,\Upsilon];V))\bigcap L^2(\Omega; L^2([0,\Upsilon];D(A)))\quad \forall T>0,
\]
and for every $t\geq 0$,
\begin{eqnarray}\label{eq-5}
 Y(t)+\int^{t}_{0}\Big(AY+B(Y,Y)+G(Y)\Big)ds=\gamma+\int^{t}_{0}\psi(s,Y(s)) dW(s),
 \end{eqnarray}
 holds in $V'$, $P-$ a.s..
\end{dfn}
In order to obtain the global well-posedness of strong solution to (\ref{equ-7}), the following Hypothesis H is required by \cite{D-G-T-Z}.

Given any pair of Banach spaces $\mathcal{X}$ and $\mathcal{Y}$,  $Bnd_u(\mathcal{X}, \mathcal{Y})$ stands for the collection of all continuous mappings $\psi: [0,\infty)\times \mathcal{X} \rightarrow \mathcal{Y}$ such that
\[
\|\psi(t,x)\|_{\mathcal{Y}}\leq L(1+\|x\|_{\mathcal{X}}) \quad x\in \mathcal{X}, \ t\geq 0,
\]
where the constant $L$ is independent of $t$. We say $\psi\in Lip_u(\mathcal{X},\mathcal{Y})$, if in addition,
\[
\|\psi(t,x)-\psi(t,y)\|_{\mathcal{Y}}\leq L\|x-y\|_{\mathcal{X}} \quad x,y\in \mathcal{X}, \ t\geq 0.
\]
\begin{description}
  \item[Hypothesis H] $\psi: [0,\infty)\times H\rightarrow \mathcal{L}_2(U; H)$ satisfies
\begin{eqnarray*}
\psi\in Lip_u(H, \mathcal{L}_2(U; H))\cap Lip_u(V, \mathcal{L}_2(U; V))\cap Bnd_u(V, \mathcal{L}_2(U; D(A))).
\end{eqnarray*}
\end{description}

The following result is given by Theorem 2.1 in \cite{D-G-T-Z}.
\begin{prp}\label{prp-1}
For any  $\mathcal{F}_0-$measurable $\gamma\in L^2(\Omega;V)$, under Hypothesis H, there exists a unique global solution $Y$ of (\ref{equ-7}) with $Y(0)=\gamma$.
\end{prp}
Moreover, we recall both Theorem 3.1 and Lemma 5.1 in \cite{D-G-T-Z} satisfied by $Y$. For the readers' convenience, we only state Theorem 3.1 in \cite{D-G-T-Z} here.
\begin{prp}\label{prp-2}
 If for some $p\geq 2$, $\gamma\in L^p(\Omega; V)$, then, under Hypothesis H, for any $\Upsilon>0$, there exists constant $C(p)$ such that
\begin{eqnarray}\label{eq-9}
E\Big(\sup_{t\in[0,\Upsilon]}|Y|^p+\int^\Upsilon_0 \|Y\|^2|Y|^{p-2}dt\Big)\leq C(p),
\end{eqnarray}
and
\begin{eqnarray}\label{eq-10}
E\Big(\int^\Upsilon_0\|Y\|^2dt\Big)^{\frac{p}{2}}\leq C(p).
\end{eqnarray}
\end{prp}
\begin{remark}\label{r-1}
If $\gamma\in L^{\infty}(V)$, then for any $p\geq 1$,
 \begin{eqnarray}\label{eq-29}
E\Big(\int^\Upsilon_0\|Y\|^2dt\Big)^{p}\leq C(p).
\end{eqnarray}
\end{remark}

From now on, throughout the whole paper, we always assume Hypothesis H holds. It's worth mentioning that no extra conditions on $\psi$ are needed to obtain the main result of this paper (see Theorem \ref{thm-1}).
\section{Small time asymptotics}
Let $\varepsilon>0$, by the scaling property of the Brownian motion, it is easy to see that $Y(\varepsilon t)$ coincides in law with the solution of the following equation:
\begin{eqnarray}\label{eq-6}
Y^{\varepsilon}(t)+\varepsilon\int^{t}_{0}\Big(AY^{\varepsilon}+B(Y^{\varepsilon},Y^{\varepsilon})+G(Y^{\varepsilon})\Big)ds=\gamma+\sqrt{\varepsilon}\int^{t}_{0}\psi(\varepsilon s,Y^{\varepsilon}(s)) dW(s).
\end{eqnarray}
Let $\mu^{\varepsilon}_{\gamma}$ be the law of $Y^{\varepsilon}(\cdot)$ on $C([0,\Upsilon];V)$ with initial data $Y^\varepsilon(0)=\gamma$. Define a functional $I(g)$ on $C([0,\Upsilon];V)$ by
\begin{eqnarray}\label{eq-33}
I(g)=\inf_{h\in\Gamma_g}\left\{\frac{1}{2}\int^\Upsilon_0|\dot{h}(t)|^2_{U}dt\right\},
\end{eqnarray}
where
\begin{eqnarray*}
\Gamma_g&=&\Big\{h\in C([0,\Upsilon];V): h(\cdot) \ {\rm{is\ absolutely\ continuous\ and\ such\ that}}\ \\
&& g(t)=\gamma+\int^t_0 \psi(s, g(s))\dot{h}(s)ds,\  0\leq t\leq \Upsilon \Big \}.
\end{eqnarray*}
The main result of this paper reads as
\begin{thm}\label{thm-1}
For any initial value $\gamma\in L^{\infty}( V)$, $\mu^{\varepsilon}_\gamma$ satisfies a large deviation principle with the rate function $I(\cdot)$ defined by (\ref{eq-33}), that is,
\begin{description}
  \item[(i)] For any closed subset $F\subset C([0,\Upsilon];V) $,
  \[
  \lim_{\varepsilon\rightarrow 0}\sup_{\gamma_n\rightarrow \gamma}\varepsilon \log \mu^{\varepsilon}_{\gamma_n}(F)\leq -\inf_{g\in F}I(g).
  \]
  \item[(ii)] For any open subset $G\subset C([0,\Upsilon];V) $,
  \[
  \lim_{\varepsilon\rightarrow 0}\inf_{\gamma_n\rightarrow \gamma}\varepsilon \log \mu^{\varepsilon}_{\gamma_n}(G)\geq -\inf_{g\in G}I(g).
  \]
\end{description}
\end{thm}

\begin{proof}
 Let $Z^{\varepsilon}=(Z^{1,\varepsilon}, Z^{2,\varepsilon})$ be the solution of the stochastic equation
\begin{eqnarray}\label{eq-7}
Z^{\varepsilon}(t)=\gamma+\sqrt{\varepsilon}\int^t_0\psi(\varepsilon s, Z^{\varepsilon}(s))dW(s),
\end{eqnarray}
and $\nu^{\varepsilon}$ be the law of $Z^{\varepsilon}(\cdot)$ on $C([0,\Upsilon];V)$. By \cite{Daprato}, we know that $\nu^{\varepsilon}$ satisfies a large deviation principle with the rate function $I(\cdot)$. Based on Theorem 4.2.13 in \cite{DZ}, we only need to show that two families of the probability measures $\mu^{\varepsilon}$ and $\nu^{\varepsilon}$ are exponentially equivalent, that is, for any $\delta >0$,
\begin{eqnarray}\label{eq-8}
\lim_{\varepsilon\rightarrow 0}\varepsilon \log P(\sup_{0\leq t\leq \Upsilon}\|Y^{\varepsilon}(t)-Z^{\varepsilon}(t)\|^2>\delta)=-\infty,
\end{eqnarray}
which will be proved in Sect. \ref{sect3.3}.
\end{proof}

\subsection{Energy estimates}\label{Sect3.2}
In order to prove  (\ref{eq-8}), we need to make some a priori estimates.
Notice that the only differences  between (\ref{eq-5}) and (\ref{eq-6}) are the constant coefficients, so Proposition \ref{prp-2} and Remark \ref{r-1} still hold for $Y^\varepsilon$.

\
The following is the main result in this part, which gives the probability of $Y^\varepsilon$ leaves an energy ball in $V$. Set
\[
(|Y^{\varepsilon}|_V(\Upsilon))^2:=\sup_{0\leq t\leq \Upsilon}\|Y^{\varepsilon}(t)\|^2+\varepsilon\int^\Upsilon_0|AY^{\varepsilon}|^2dt.
\]
Then, we claim that
\begin{thm}\label{thm-2}
\begin{eqnarray}\label{eeee-2}
\lim_{M\rightarrow \infty}\sup_{0<\varepsilon\leq 1}\varepsilon\log P\Big((|Y^{\varepsilon}|_V(\Upsilon))^2>M\Big)=-\infty.
\end{eqnarray}
\end{thm}
\

It's difficult to prove Theorem  \ref{thm-2} directly like 2D Navier-Stokes equations since the nonlinear terms of 3D primitive equations have no cancellation property in $V$. To overcome this difficulty, we introduce stopping times $\tau^{(1)}_K$, $\tau^{(2)}_K$  and $\tau_K$. Further,  we verify Theorem \ref{thm-2} holds for $|Y^{\varepsilon}|_V(\Upsilon\wedge \tau_K)$ (see Proposition \ref{lem-1}). To achieve the result of Proposition \ref{lem-1}, some additional exponential moment estimates of $\tau^{(1)}_K$ and $\tau^{(2)}_K$ are required (see Proposition \ref{lem-2} and Proposition \ref{lem-3}).
\

For some constant $K>0$, define the following stopping times
\begin{eqnarray*}
\tau_K&:=&\tau^{(1)}_K\wedge \tau^{(2)}_K,\\
\tau^{(1)}_K&:= &\inf\Big\{t: |\mathbf{v}^{\varepsilon}|^4_{4}> K\Big\},\\
\tau^{(2)}_K&:= &\inf\Big\{t: |\partial_z Y^{\varepsilon}|^2>K,\ {\rm{or}}\ \varepsilon\int^t_0\|\partial_z Y^{\varepsilon}\|^2ds> K\Big\}.
\end{eqnarray*}
Note that
\begin{eqnarray}\notag
&&\varepsilon\log P\Big((|Y^{\varepsilon}|_V(\Upsilon))^2>M\Big)\\ \notag
&\leq& \varepsilon\log P\Big((|Y^{\varepsilon}|_V(\Upsilon))^2>M,\sup_{t\in[0,\Upsilon]}|\mathbf{v}^{\varepsilon}|^4_4\leq K,\ \sup_{t\in[0,\Upsilon]}|\partial_z Y^{\varepsilon}|^2+\varepsilon\int^\Upsilon_0 \|\partial_z Y^{\varepsilon}\|^2dt\leq K\Big)\\ \notag
&&+ \varepsilon\log P\Big((|Y^{\varepsilon}|_V(\Upsilon))^2>M,\sup_{t\in[0,\Upsilon]}|\mathbf{v}^{\varepsilon}|^4_4> K\Big)\\ \notag
&& + \varepsilon\log P\Big((|Y^{\varepsilon}|_V(\Upsilon))^2>M,\sup_{t\in[0,\Upsilon]}|\partial_z Y^{\varepsilon}|^2+\varepsilon\int^\Upsilon_0 \|\partial_z Y^{\varepsilon}\|^2dt> K\Big)\\ \notag
&\leq& \varepsilon\log P\Big((|Y^{\varepsilon}|_V(\Upsilon\wedge \tau_K))^2>M\Big)+\varepsilon\log P\Big(\sup_{t\in[0,\Upsilon]}|\mathbf{v}^{\varepsilon}|^4_4> K\Big)\\
\label{e-10}
&&+\varepsilon\log P\Big(\sup_{t\in[0,\Upsilon]}|\partial_z Y^{\varepsilon}|^2+\varepsilon\int^\Upsilon_0 \|\partial_z Y^{\varepsilon}\|^2dt> K\Big).
\end{eqnarray}
Thus, in order to establish Theorem \ref{thm-2}, we need to prove
\begin{eqnarray}\label{eeee-3}
\lim_{M\rightarrow \infty}\sup_{0<\varepsilon\leq1}\varepsilon\log P\Big((|Y^{\varepsilon}|_V(\Upsilon\wedge \tau_K))^2>M\Big)=-\infty,
\end{eqnarray}
\begin{eqnarray}\label{eq-12}
\lim_{K\rightarrow \infty}\sup_{0<\varepsilon\leq1}\varepsilon\log P\Big(\sup_{t\in[0,\Upsilon]}|\mathbf{v}^{\varepsilon}|^4_4>K\Big)=-\infty,
\end{eqnarray}
and
\begin{eqnarray}\label{eq-13}
\lim_{K\rightarrow \infty}\sup_{0<\varepsilon\leq1}\varepsilon\log P\Big(\sup_{t\in[0,\Upsilon]}|\partial_z Y^{\varepsilon}|^2+\varepsilon\int^\Upsilon_0 \|\partial_z Y^{\varepsilon}\|^2dt> K\Big)=-\infty.
\end{eqnarray}
The above (\ref{eeee-3})-(\ref{eq-13}) will be proved by the following Proposition \ref{lem-1} - Proposition \ref{lem-3}, respectively.

Firstly, for (\ref{eeee-3}),
\begin{prp}\label{lem-1}
\begin{eqnarray}
\lim_{M\rightarrow \infty}\sup_{0<\varepsilon\leq1}\varepsilon\log P\Big((|Y^{\varepsilon}|_V(\Upsilon\wedge \tau_K))^2>M\Big)=-\infty.
\end{eqnarray}
\end{prp}
\begin{proof}
 Applying It\^{o} formula to $\|Y^{\varepsilon}\|^2$, we deduce that
\begin{eqnarray*}
&&d\|Y^{\varepsilon}\|^2+2\varepsilon|AY^{\varepsilon}|^2dt\\
&=&-2\varepsilon\langle B(Y^{\varepsilon},Y^{\varepsilon} ), AY^{\varepsilon}\rangle dt
-2\varepsilon\langle G(Y^{\varepsilon}), AY^{\varepsilon}\rangle dt\\
&& +\varepsilon \|\psi(\varepsilon t, Y^{\varepsilon})\|^2_{\mathcal{L}_2(U;V)}dt
+2\sqrt{\varepsilon}\langle A^{\frac{1}{2}}\psi(\varepsilon t, Y^{\varepsilon})dW, \ A^{\frac{1}{2}}Y^{\varepsilon} \rangle.
\end{eqnarray*}
Referring to Theorem 3.2 in \cite{D-G-T-Z}, it gives that
\[
|\langle B(Y^{\varepsilon},Y^{\varepsilon} ), AY^{\varepsilon}\rangle|
\leq \frac{1}{4}|AY^{\varepsilon}|^2+C(|\mathbf{v}^{\varepsilon}|^8_{4}+|\partial_z Y^{\varepsilon}|^2\|\partial_z Y^{\varepsilon}\|^2)\|Y^{\varepsilon}\|^2.
\]
By the Cauchy-Schwarz inequality and the Young's inequality, we have
\[
|\langle G(Y^{\varepsilon}), AY^{\varepsilon}\rangle |
\leq \frac{1}{4}|AY^{\varepsilon}|^2+C(1+\|Y^{\varepsilon}\|^2).
\]
Then, we deduce from Hypothesis H that
\begin{eqnarray*}
&&\|Y^{\varepsilon}(t)\|^2+\varepsilon\int^t_0|AY^{\varepsilon}(s)|^2ds\\
&\leq& \|\gamma\|^2+2\varepsilon Ct+\varepsilon L^2t+2\varepsilon C\int^t_0(|\mathbf{v}^{\varepsilon}|^8_4
+|\partial_z Y^{\varepsilon}|^2\|\partial_z Y^{\varepsilon}\|^2+1)\|Y^{\varepsilon}\|^2ds\\
&&+\varepsilon L^2\int^t_0\|Y^{\varepsilon}\|^2ds
+2\sqrt{\varepsilon}|\int^t_0\langle A^{\frac{1}{2}}\psi(\varepsilon s, Y^{\varepsilon})dW, \ A^{\frac{1}{2}}Y^{\varepsilon} \rangle|.
\end{eqnarray*}
In view of the definition of $\tau_K$, we have
\begin{eqnarray}\notag
&&(|Y^{\varepsilon}|_V(\Upsilon\wedge \tau_K))^2\\ \notag
&\leq& \|\gamma\|^2+2\varepsilon C\Upsilon+\varepsilon L^2\Upsilon+2\varepsilon C\int^{\Upsilon\wedge \tau_K}_0(K^2
+K\|\partial_z Y^{\varepsilon}\|^2+1)(|Y^{\varepsilon}|_V(s))^2ds\\
\label{e-2}
&&+\varepsilon L^2\int^{\Upsilon\wedge \tau_K}_0(|Y^{\varepsilon}|_V(s))^2ds
+2\sqrt{\varepsilon}\sup_{t\in[0,\Upsilon\wedge \tau_K]}|\int^t_0\langle A^{\frac{1}{2}}\psi(\varepsilon s, Y^{\varepsilon})dW,\  A^{\frac{1}{2}}Y^{\varepsilon} \rangle|.
\end{eqnarray}
Applying Gronwall inequality to (\ref{e-2}), we have
\begin{eqnarray*}
(|Y^{\varepsilon}|_V(\Upsilon\wedge \tau_K))^2
&\leq& \Big(\|\gamma\|^2+2\varepsilon C\Upsilon+\varepsilon L^2\Upsilon+2\sqrt{\varepsilon}\sup_{t\in[0,\Upsilon\wedge \tau_K]}|\int^t_0\langle A^{\frac{1}{2}}\psi(\varepsilon s, Y^{\varepsilon})dW, \ A^{\frac{1}{2}}Y^{\varepsilon} \rangle|\Big)\\
&&\cdot \exp\Big\{2\varepsilon CK^2 \Upsilon+2\varepsilon C K\int^{\Upsilon\wedge \tau_K}_0\|\partial_z Y^{\varepsilon}\|^2ds+2\varepsilon C\Upsilon+\varepsilon L^2\Upsilon\Big\}\\
&\leq& \Big(\|\gamma\|^2+2\varepsilon C\Upsilon+\varepsilon L^2\Upsilon+2\sqrt{\varepsilon}\sup_{t\in[0,\Upsilon\wedge \tau_K]}|\int^t_0\langle A^{\frac{1}{2}}\psi(\varepsilon s, Y^{\varepsilon})dW, \ A^{\frac{1}{2}}Y^{\varepsilon} \rangle|\Big)\\
&&\cdot \exp\Big\{2\varepsilon CK^2 \Upsilon+2\varepsilon C K^2+2\varepsilon C\Upsilon+\varepsilon L^2\Upsilon\Big\}.
\end{eqnarray*}
Hence, by H\"{o}lder inequality, we have for $p\geq 2$,
\begin{eqnarray}\notag
&&(E(|Y^{\varepsilon}|_V(\Upsilon\wedge \tau_K))^{2p})^{\frac{1}{p}}\\ \notag
&\leq& \Big[\|\gamma\|^2+2\varepsilon C\Upsilon+\varepsilon L^2\Upsilon+2\sqrt{\varepsilon}\Big(E\big(\sup_{t\in[0,\Upsilon\wedge \tau_K]}|\int^t_0\langle A^{\frac{1}{2}}\psi(\varepsilon s, Y^{\varepsilon})dW,\ A^{\frac{1}{2}}Y^{\varepsilon} \rangle|^p\big)\Big)^{\frac{1}{p}}\Big]\\
\label{e-3}
&&\cdot \exp\Big\{2\varepsilon CK^2 \Upsilon+2\varepsilon C K^2+2\varepsilon C\Upsilon+\varepsilon L^2\Upsilon\Big\}.
\end{eqnarray}
To estimate the stochastic integral term in (\ref{e-3}), we will use the following remarkable result from \cite{B-Y,Davis} that there exists a universal constant $C$ such that, for any $p\geq 2$ and for any continuous martingale $M_t$ with $M_0=0$,
\begin{eqnarray}\label{eq-11}
(E(|M^*_t|^p))^{\frac{1}{p}}\leq Cp^{\frac{1}{2}}(E\langle M\rangle^{\frac{p}{2}}_t)^{\frac{1}{p}},
\end{eqnarray}
where $M^*_t=\sup_{s\in [0,t]}|M_s|$.

Using (\ref{eq-11}), we deduce that
\begin{eqnarray}\notag
&&\Big(E\big(\sup_{t\in[0,\Upsilon\wedge \tau_K]}|\int^t_0\langle A^{\frac{1}{2}}\psi(\varepsilon s, Y^{\varepsilon})dW, \ A^{\frac{1}{2}}Y^{\varepsilon} \rangle|\big)^p\Big)^{\frac{1}{p}}\\ \notag
&\leq& C\sqrt{p}\Big(E(\int^{\Upsilon\wedge \tau_K}_0 \|Y^{\varepsilon}\|^2\|\psi(\varepsilon s,Y^\varepsilon)\|^2_{\mathcal{L}_2(U;V)} ds)^{\frac{p}{2}}\Big)^{\frac{1}{p}}\\ \notag
&\leq& C\sqrt{p}L\Big(E\big(\int^{\Upsilon\wedge \tau_K}_0\|Y^{\varepsilon}\|^2(1+\|Y^{\varepsilon}\|^2) ds\big)^{\frac{p}{2}}\Big)^{\frac{1}{p}}\\ \notag
&\leq& C\sqrt{p}L\Big[\Big(E\big(\int^{\Upsilon\wedge \tau_K}_0(1+\|Y^{\varepsilon}\|^2)^2 ds\big)^{\frac{p}{2}}\Big)^{\frac{2}{p}}\Big]^{\frac{1}{2}}\\ \notag
&\leq& C\sqrt{p}L\Big[\Big(E\big(\int^{\Upsilon\wedge \tau_K}_0(1+\|Y^{\varepsilon}\|^4) ds\big)^{\frac{p}{2}}\Big)^{\frac{2}{p}}\Big]^{\frac{1}{2}}\\
\label{e-4}
&\leq& C\sqrt{p}L\Big[\int^{\Upsilon\wedge \tau_K}_01+(E\|Y^{\varepsilon}(s)\|^{2p})^{\frac{2}{p}}ds\Big]^{\frac{1}{2}}.
\end{eqnarray}
As a result of (\ref{e-3}) and (\ref{e-4}), we have
\begin{eqnarray}\notag
&&(E(|Y^{\varepsilon}|_V(\Upsilon\wedge \tau_K))^{2p})^{\frac{2}{p}}\\ \notag
&\leq& \Big[(\|\gamma\|^2+2\varepsilon C\Upsilon+\varepsilon L^2\Upsilon)^2+4C\varepsilon p\Upsilon L^2+4C{\varepsilon p}L^2\int^{\Upsilon\wedge \tau_K}_0(E(|Y^{\varepsilon}|_V(t))^{2p})^{\frac{2}{p}}dt\Big]\\
\label{e-5}
&&\cdot \exp\Big\{4\varepsilon CK^2 \Upsilon+4\varepsilon C K^2+4\varepsilon C\Upsilon+2\varepsilon L^2\Upsilon\Big\}.
\end{eqnarray}
Applying Gronwall inequality to (\ref{e-5}), we get
\begin{eqnarray*}
&&(E(|Y^{\varepsilon}|_V(\Upsilon\wedge \tau_K))^{2p})^{\frac{2}{p}}\\
&\leq& C\exp\Big\{4\varepsilon CK^2 \Upsilon+4\varepsilon C K^2+4\varepsilon C\Upsilon+2\varepsilon L^2\Upsilon\Big\}\\
&&\cdot \Big[(\|\gamma\|^2+2\varepsilon C\Upsilon+\varepsilon L^2\Upsilon)^2+4C\varepsilon p\Upsilon L^2\Big]\cdot \exp\Big\{4C{\varepsilon pL^2\Upsilon}\cdot{\rm{e}}^{4\varepsilon CK^2 \Upsilon+4\varepsilon C K^2+4\varepsilon C\Upsilon+2\varepsilon L^2\Upsilon}\Big\}.
\end{eqnarray*}
Taking $p=\frac{1}{\varepsilon}$ and utilizing the Chebyshev inequality, we obtain
 \begin{eqnarray}\notag
&&\varepsilon \log P((|Y^{\varepsilon}|_V(\Upsilon\wedge \tau_K))^2>M)\\ \notag
&\leq& -\log M+\log(E(|Y^{\varepsilon}|_V(\Upsilon\wedge \tau_K))^{2p})^{\frac{1}{p}}\\ \notag
&\leq& -\log M+\log \sqrt{(\|\gamma\|^2+2\varepsilon C\Upsilon+\varepsilon L^2\Upsilon)^2+4C\varepsilon p\Upsilon L^2} +2\varepsilon CK^2 \Upsilon+2\varepsilon C K^2+2\varepsilon C\Upsilon\\
\label{e-1}
&&+\varepsilon L^2\Upsilon+2C{\varepsilon pL^2\Upsilon}\cdot{\rm{e}}^{4\varepsilon CK^2 \Upsilon+4\varepsilon C K^2+4\varepsilon C\Upsilon+2\varepsilon L^2\Upsilon}.
\end{eqnarray}
Let $M\rightarrow \infty$ on both sides of (\ref{e-1}), we complete the proof.
\end{proof}

 For (\ref{eq-12}), we have
\begin{prp}\label{lem-2}
\begin{eqnarray*}
\lim_{K\rightarrow \infty}\sup_{0<\varepsilon\leq1}\varepsilon\log P\Big(\sup_{t\in[0,\Upsilon]}|\mathbf{v}^{\varepsilon}|^4_4>K\Big)=-\infty.
\end{eqnarray*}
\end{prp}
In order to prove Proposition \ref{lem-2}, we prove the following Lemma \ref{lem-4}- Lemma \ref{lem-9}. Now, we introduce the following auxiliary process $\breve{Y}^{\varepsilon}$ satisfying
\begin{eqnarray}\label{e-6}
\left\{
  \begin{array}{ll}
    d\breve{Y}^{\varepsilon}+\varepsilon A\breve{Y}^{\varepsilon}dt=\sqrt{\varepsilon}\psi(\varepsilon t,{Y}^{\varepsilon})dW,  \\
    \breve{Y}^{\varepsilon}(0)=0.
  \end{array}
\right.
\end{eqnarray}
Then, we have
\begin{lemma}\label{lem-4}
For any $p\geq 1$,
\begin{eqnarray}\label{eq-17}
\lim_{K\rightarrow \infty}\sup_{0<\varepsilon\leq1}\varepsilon\log P\Big(\sup_{t\in[0,\Upsilon]}\|\breve{Y}^{\varepsilon}(t)\|^{2p}> K\Big)=-\infty.
\end{eqnarray}
\end{lemma}
\begin{proof}
 Applying It\^{o} formula to $\| \breve{Y}^{\varepsilon}(t)\|^2$, we have
\begin{eqnarray*}
\| \breve{Y}^{\varepsilon}(t)\|^2+2\varepsilon\int^t_0|A \breve{Y}^{\varepsilon}(s)|^2ds
=2\sqrt{\varepsilon}\int^t_0\langle A^{\frac{1}{2}} \breve{Y}^{\varepsilon}, A^{\frac{1}{2}}\psi(\varepsilon s, \ {Y}^{\varepsilon})dW(s)\rangle
+\varepsilon \int^t_0\|\psi(\varepsilon s,  {Y}^{\varepsilon})\|^2_{\mathcal{L}_2(U;V)}ds.
\end{eqnarray*}
By Hypothesis H, we get
\begin{eqnarray*}
\| \breve{Y}^{\varepsilon}(t)\|^2+\varepsilon\int^t_0|A \breve{Y}^{\varepsilon}(s)|^2ds
\leq 2\sqrt{\varepsilon}\int^t_0\langle A^{\frac{1}{2}} \breve{Y}^{\varepsilon},A^{\frac{1}{2}} \psi(\varepsilon s, \ {Y}^{\varepsilon})dW(s)\rangle
+\varepsilon L^2t+\varepsilon L^2\int^t_0\|{Y}^{\varepsilon}(s)\|^2ds.
\end{eqnarray*}
Let
\[
M_t:=2\sqrt{\varepsilon}\int^t_0\langle A^{\frac{1}{2}} \breve{Y}^{\varepsilon}, A^{\frac{1}{2}}\psi(\varepsilon s,  {Y}^{\varepsilon})dW(s)\rangle,\quad M^*_T:=\sup_{t\in[0,\Upsilon]}|M(t)|.
\]
By H\"{o}lder inequality, we have for any $p\geq 1$,
\begin{eqnarray}\label{e-7}
(E\sup_{t\in[0,\Upsilon]}\|\breve{Y}^{\varepsilon}(t)\|^{2p})^{\frac{1}{p}}\leq (E|M^*_\Upsilon|^{p})^{\frac{1}{p}}+\varepsilon L^2 \Upsilon+\varepsilon L^2 \Big(E\big(\int^\Upsilon_0\|{Y}^{\varepsilon}\|^2ds\big)^p\Big)^{\frac{1}{p}}.
\end{eqnarray}
For $(E|M^*_\Upsilon|^{p})^{\frac{1}{p}}$ in (\ref{e-7}), by (\ref{eq-11}), we have
\begin{eqnarray}\notag
(E|M^*_\Upsilon|^{p})^{\frac{1}{p}}
&\leq& 2C\sqrt{\varepsilon p}L\Big(E\big(\int^\Upsilon_0\|\breve{Y}^{\varepsilon}\|^2(1+\|{Y}^{\varepsilon}\|^2)ds\big)^{\frac{p}{2}}\Big)^{\frac{1}{p}}
\\ \notag
&\leq& 2C\sqrt{\varepsilon p}L\Big[E\Big(\eta\sup_{t\in [0,\Upsilon]}\|\breve{Y}^{\varepsilon}\|^4+C(\int^\Upsilon_0(1+\|{Y}^{\varepsilon}\|^2)ds)^2\Big)^{\frac{p}{2}}\Big]^{\frac{1}{p}}
\\ \notag
&\leq& 2C\sqrt{\varepsilon p}L\eta(E\sup_{t\in [0,\Upsilon]}\|\breve{Y}^{\varepsilon}\|^{2p})^{\frac{1}{p}}+2C\sqrt{\varepsilon p}L\Big(E\big(\int^\Upsilon_0(1+\|{Y}^{\varepsilon}\|^2)ds\big)^p\Big)^{\frac{1}{p}}\\
\label{eq-16}
&\leq& \frac{1}{2}(E\sup_{t\in [0,\Upsilon]}\|\breve{Y}^{\varepsilon}\|^{2p})^{\frac{1}{p}}+2C\sqrt{\varepsilon p}L\Big(E\big(\int^\Upsilon_0(1+\|{Y}^{\varepsilon}\|^2)ds\big)^p\Big)^{\frac{1}{p}},
\end{eqnarray}
where $\eta $ is chosen to be sufficiently small such that $2C\sqrt{\varepsilon p}L\eta<\frac{1}{2}$.

As a result of (\ref{e-7}) and (\ref{eq-16}), we have
\begin{eqnarray*}\notag
&&(E\sup_{t\in[0,\Upsilon]}\|\breve{Y}^{\varepsilon}(t)\|^{2p})^{\frac{1}{p}}\\
&\leq& 2\varepsilon L\Upsilon+2\varepsilon L \Big(E(\int^\Upsilon_0\|{Y}^{\varepsilon}\|^2ds)^p\Big)^{\frac{1}{p}}+4C\sqrt{\varepsilon p}L\Big(E(\int^\Upsilon_0(1+\|{Y}^{\varepsilon}\|^2)ds)^p\Big)^{\frac{1}{p}}.
\end{eqnarray*}
By (\ref{eq-29}), we obtain
\begin{eqnarray}\label{eq-31}
(E\sup_{t\in[0,\Upsilon]}\|\breve{Y}^{\varepsilon}(t)\|^{2p})^{\frac{1}{p}}
\leq 2\varepsilon L\Upsilon+2\varepsilon C(p) L +4C(p)\sqrt{\varepsilon p}L.
\end{eqnarray}
Using the same argument as Proposition \ref{lem-1}, we conclude the result.
\end{proof}

\

Applying It\^{o} formula to $A\breve{Y}^{\varepsilon}(t)$ and using Hypothesis H, we can easily show that
\begin{lemma}\label{lem-10}
For any $p\geq 2$,
\begin{eqnarray*}
E\sup_{t\in[0,\Upsilon]}|A\breve{Y}^{\varepsilon}(t)|^{2p}\leq C(p),
\end{eqnarray*}
where $C(p)$ is independent of $\varepsilon$.
\end{lemma}

\

Let $\hat{Y}^{\varepsilon}:={Y}^{\varepsilon}-\breve{Y}^{\varepsilon}=(\hat{\mathbf{v}}^{\varepsilon},\hat{T}^{\varepsilon}) $.
From (\ref{eq-6}) and (\ref{e-6}), we deduce that $\hat{Y}^{\varepsilon}$ satisfies
\begin{eqnarray}\label{e-6-1}
\left\{
  \begin{array}{ll}
    d\hat{Y}^{\varepsilon}+\varepsilon A\hat{Y}^{\varepsilon}dt+\varepsilon B(\hat{Y}^{\varepsilon}+\breve{Y}^{\varepsilon},\hat{Y}^{\varepsilon}+\breve{Y}^{\varepsilon} )dt+\varepsilon G(\hat{Y}^{\varepsilon}+\breve{Y}^{\varepsilon})dt=0,  \\
    \hat{Y}^{\varepsilon}(0)=\gamma.
  \end{array}
\right.
\end{eqnarray}

\begin{lemma}\label{lem-9}
For any $p\geq 1$,
\begin{eqnarray}\label{eq-15}
\lim_{K\rightarrow \infty}\sup_{0<\varepsilon\leq1}\varepsilon\log P\Big(|\hat{\mathbf{v}}^{\varepsilon}(t)|^{4p}_4>K\Big)=-\infty.
\end{eqnarray}
\end{lemma}
\begin{proof}
Referring to Proposition 4.3 in \cite{D-G-T-Z}, we obtain
\begin{eqnarray*}
\sup_{t\in[0,\Upsilon]}|\hat{\mathbf{v}}^{\varepsilon}(t)|^4_4\leq C\|\gamma\|^4+C\varepsilon(1+\sup_{t\in[0,\Upsilon]}|\hat{\mathbf{v}}^{\varepsilon}(t)|^4_4+\sup_{t\in[0,\Upsilon]}|\breve{\mathbf{v}}^{\varepsilon}(t)|^4_4)
\int^{\Upsilon}_0H(t)dt,
\end{eqnarray*}
where
$
H(t)=(1+\|Y^{\varepsilon}(t)\|^2)(1+|A\breve{Y}^\varepsilon(t)|^4)
$
satisfying  $H(t)\in L^1([0,\Upsilon])$.

Utilizing Proposition A.2 in \cite{D-G-T-Z}, it gives that
\begin{eqnarray*}
\sup_{t\in[0,\Upsilon]}|\hat{\mathbf{v}}^{\varepsilon}(t)|^4_4\leq C\|\gamma\|^4+C\varepsilon(1+\sup_{t\in[0,\Upsilon]}|\breve{\mathbf{v}}^{\varepsilon}(t)|^4_4)
\int^{\Upsilon}_0H(t)dt.
\end{eqnarray*}
By H\"{o}lder inequality, we deduce that for any $p\geq1$,
\begin{eqnarray}\label{eee-2}
(E\sup_{t\in[0,\Upsilon]}|\hat{\mathbf{v}}^{\varepsilon}(t)|^{4p}_4)^{\frac{1}{p}}\leq C\varepsilon+C\|\gamma\|^4+C\varepsilon(E\sup_{t\in[0,\Upsilon]}\|\breve{\mathbf{v}}^{\varepsilon}(t)\|^{8p})^{\frac{1}{p}}+
C\varepsilon\Big(E(\int^{\Upsilon}_0H(t)dt)^{2p}\Big)^{\frac{1}{p}}.
\end{eqnarray}
With the aid of the Young's inequality, we have
\begin{eqnarray}\notag
\Big(E\big(\int^{\Upsilon}_0H(t)dt\big)^{2p}\Big)^{\frac{1}{p}}
&\leq& \Big(E\big(\int^{\Upsilon}_0(1+\|Y^\varepsilon (t)\|^2)(1+|A\breve{Y}^\varepsilon|^4)dt\big)^{2p}\Big)^{\frac{1}{p}}\\ \notag
&\leq& \Big(E\big(\sup_{t\in [0,\Upsilon]}(1+|A\breve{Y}^\varepsilon|^4)\int^{\Upsilon}_0(1+\|Y^\varepsilon (t)\|^2)dt\big)^{2p}\Big)^{\frac{1}{p}}\\
\label{eee-1}
&\leq&C\Big(E\sup_{t\in [0,\Upsilon]}(1+|A\breve{Y}^\varepsilon|^4)^{4p}\Big)^{\frac{1}{p}}+C\Big(E\big(\int^{\Upsilon}_0(1+\|Y^\varepsilon (t)\|^2)dt\big)^{4p}\Big)^{\frac{1}{p}}.
\end{eqnarray}
Based on (\ref{eee-2}) and (\ref{eee-1}), we get
\begin{eqnarray*}\notag
(E\sup_{t\in[0,\Upsilon]}|\hat{\mathbf{v}}^{\varepsilon}(t)|^{4p}_4)^{\frac{1}{p}}
&\leq &C\varepsilon+C\|\gamma\|^4+C\varepsilon(E\sup_{t\in[0,\Upsilon]}\|\breve{\mathbf{v}}^{\varepsilon}(t)\|^{8p})^{\frac{1}{p}}+
C\varepsilon\Big(E\sup_{t\in [0,\Upsilon]}(1+|A\breve{Y}^\varepsilon|^4)^{4p}\Big)^{\frac{1}{p}}\\
&&+C\varepsilon\Big(E\big(\int^\Upsilon_0(1+\|Y^\varepsilon (t)\|^2)dt\big)^{4p}\Big)^{\frac{1}{p}}.
\end{eqnarray*}
By (\ref{eq-29}) and Lemma \ref{lem-10}, we have
\begin{eqnarray}\label{eq-30}
(E\sup_{t\in[0,\Upsilon]}|\hat{\mathbf{v}}^{\varepsilon}(t)|^{4p}_4)^{\frac{1}{p}}
\leq C\|\gamma\|^4+C(p)\varepsilon.
\end{eqnarray}
Using the same argument as the proof of Proposition \ref{lem-1}, we conclude the result.
\end{proof}

Based on Lemma \ref{lem-4}-Lemma \ref{lem-9}, we are ready to prove Proposition \ref{lem-2}.

\begin{flushleft}
\textbf{Proof of Proposition \ref{lem-2}.} \quad
 Since $H^1(\mathcal{O})$ is embedded in $L^4(\mathcal{O})$, we have
\[
\varepsilon \log P\Big(\sup_{t\in[0,\Upsilon]}|\mathbf{v}^{\varepsilon}|^4_4>K\Big)\leq \varepsilon \log P\Big( \sup_{t\in[0,\Upsilon]}\|\breve{Y}^{\varepsilon}(t)\|^4 +\sup_{t\in[0,\Upsilon]}|\hat{\mathbf{v}}^\varepsilon(t)|^4_4>K \Big).
\]
Thus, it suffices to show that
\begin{eqnarray*}
\lim_{K\rightarrow \infty}\sup_{0<\varepsilon\leq1}\varepsilon\log P\Big(\sup_{t\in[0,\Upsilon]}\|\breve{Y}^{\varepsilon}(t)\|^4 +\sup_{t\in[0,\Upsilon]}|\hat{\mathbf{v}}^\varepsilon(t)|^4_4>K\Big)=-\infty.
\end{eqnarray*}
Notice that
\begin{eqnarray*}
&&(E(\sup_{t\in[0,\Upsilon]}|\hat{\mathbf{v}}^\varepsilon(t)|^4_4+\sup_{t\in[0,\Upsilon]}\|\breve{Y}^{\varepsilon}(t)\|^4)^p)^{\frac{1}{p}} \\
&\leq& C(E\sup_{t\in[0,\Upsilon]}|\hat{\mathbf{v}}^\varepsilon(t)|^{4p}_4)^{\frac{1}{p}}
+C(E\sup_{t\in[0,\Upsilon]}\|\breve{Y}^{\varepsilon}(t)\|^{4p})^{\frac{1}{p}}.
\end{eqnarray*}
By (\ref{eq-31}), (\ref{eq-30}) and using the same argument as Proposition \ref{lem-1}, we complete the proof.
\end{flushleft}

$\hfill\blacksquare$

Now, we aim to prove (\ref{eq-13}). Let $\partial_z Y^{\varepsilon}=(\partial_z \mathbf{v}^{\varepsilon}, \partial_z T^{\varepsilon})$. Applying $\partial_z$  to (\ref{eq-6}), it follows that
 \begin{eqnarray*}
&d \partial_z \mathbf{v}^\varepsilon+\varepsilon\partial_z[(\mathbf{v}^\varepsilon\cdot \nabla)\mathbf{v}^\varepsilon+\Phi(\mathbf{v}^\varepsilon)\frac{\partial \mathbf{v}^\varepsilon}{\partial z}]dt+\varepsilon (f{k}\times \partial_z\mathbf{v}^\varepsilon -\nabla T^\varepsilon) dt +\varepsilon L_1 \partial_z\mathbf{v}^\varepsilon dt =\sqrt{\varepsilon}\partial_z\psi_1(\varepsilon t,Y^\varepsilon)dW_1(t),&\\
&d \partial_z T^\varepsilon+\varepsilon\partial_z[(\mathbf{v}^\varepsilon\cdot\nabla)T^\varepsilon+\Phi(\mathbf{v}^\varepsilon)\frac{\partial T^\varepsilon}{\partial z}]dt+ \varepsilon L_2\partial_zT^\varepsilon dt=\sqrt{\varepsilon}\partial_z\psi_2(\varepsilon t,Y^\varepsilon)dW_2(t),&\\
&\int^{0}_{-1}\nabla\cdot \mathbf{v}^\varepsilon  dz=0.&
\end{eqnarray*}
We claim that
\begin{prp}\label{lem-3}
\begin{eqnarray*}
\lim_{K\rightarrow \infty}\sup_{0<\varepsilon\leq1}\varepsilon\log P\Big(\sup_{t\in[0,\Upsilon]}|\partial_z Y^{\varepsilon}|^2+\varepsilon \int^\Upsilon_0 \|\partial_z Y^{\varepsilon}\|^2dt> K\Big)=-\infty.
\end{eqnarray*}
\end{prp}
\begin{proof}
Note that there exists a constant $C_p$ independent of $\varepsilon$ such that
\begin{eqnarray*}
&&\Big(E\big(\sup_{t\in[0,\Upsilon]}|\partial_z Y^{\varepsilon}|^2+\varepsilon\int^\Upsilon_0 \|\partial_z Y^{\varepsilon}\|^2dt\big)^p\Big)^{\frac{1}{p}} \\
&\leq&
C_p\Big(E\big(\sup_{t\in[0,\Upsilon]}|\partial_z \mathbf{v}^{\varepsilon}|^2+\varepsilon\int^\Upsilon_0 \|\partial_z \mathbf{v}^{\varepsilon}\|^2dt\big)^p\Big)^{\frac{1}{p}}+C_p\Big(E\big(\sup_{t\in[0,\Upsilon]}|\partial_z T^{\varepsilon}|^2+\varepsilon\int^\Upsilon_0 \|\partial_z T^{\varepsilon}\|^2dt\big)^p\Big)^{\frac{1}{p}}\\
&:=&C_p( I_1+I_2).
\end{eqnarray*}
Let the initial value $y=(\gamma_1,\gamma_2)$. Referring to Proposition 5.2 in \cite{D-G-T-Z}, we have
\begin{eqnarray*}
&&\sup_{t\in[0,\Upsilon]}|\partial_z \mathbf{v}^{\varepsilon}|^2+\varepsilon\int^\Upsilon_0 \|\partial_z \mathbf{v}^{\varepsilon}\|^2dt\\
&\leq& C|\partial_z \gamma_1|^2+\varepsilon\int^\Upsilon_0|\mathbf{v}^{\varepsilon}|^8_4|\partial_z \mathbf{v}^{\varepsilon}|^2dt
+C\varepsilon\int^\Upsilon_0(1+\|Y^{\varepsilon}\|^2)dt\\
&&+2\sqrt{\varepsilon}\sup_{t\in[0,\Upsilon]}|\int^t_0\langle \partial_z \psi_1(\varepsilon s,Y^{\varepsilon})dW,\partial_z \mathbf{v}^{\varepsilon} \rangle|\\
&=&C|\partial_z \gamma_1|^2+\varepsilon\int^\Upsilon_0|\mathbf{v}^{\varepsilon}|^8_4|\partial_z \mathbf{v}^{\varepsilon}|^2dt
+C\varepsilon\int^\Upsilon_0(1+\|Y^{\varepsilon}\|^2)dt+2\sqrt{\varepsilon}M^{\mathbf{v}*}_\Upsilon,
\end{eqnarray*}
where
\[
M^{\mathbf{v}*}_\Upsilon=\sup_{t\in[0,\Upsilon]}|\int^t_0\langle \partial_z \psi_1(\varepsilon s, Y^{\varepsilon})dW(s), \ \partial_z \mathbf{v}^{\varepsilon} \rangle|.
\]
Similar to the proof of  (\ref{eq-16}), we obtain
\begin{eqnarray}\label{e-8}
(E|M^{\mathbf{v}*}_\Upsilon|^p)^{\frac{1}{p}}\leq2C\sqrt{\varepsilon p}L\eta\Big(E\sup_{t\in[0,\Upsilon]}|\partial_z \mathbf{v}^{\varepsilon}|^{2p}\Big)^{\frac{1}{p}}+2C\sqrt{\varepsilon p}L\Big(E\big(\int^\Upsilon_0(1+|\partial_z Y^\varepsilon|^2)dt\big)^p\Big)^{\frac{1}{p}}.
\end{eqnarray}
Then, by H\"{o}lder inequality and (\ref{e-8}), it follows that
\begin{eqnarray*}
&&\Big(E\big(\sup_{t\in[0,\Upsilon]}|\partial_z \mathbf{v}^{\varepsilon}|^2+\varepsilon\int^\Upsilon_0 \|\partial_z \mathbf{v}^{\varepsilon}\|^2dt\big)^p\Big)^{\frac{1}{p}}\\
&\leq& C|\partial_z \gamma_1|^2+C\varepsilon(E\sup_{t\in[0,\Upsilon]}|\mathbf{v}^{\varepsilon}|^{16p}_4)^{\frac{1}{p}}+C\varepsilon\Big(E\big(\int^\Upsilon_0|\partial_z \mathbf{v}^{\varepsilon}|^2dt\big)^{2p}\Big)^{\frac{1}{p}}
\\
&&+C\varepsilon\Big(E\big(\int^\Upsilon_0(1+\|Y^{\varepsilon}\|^2)dt\big)^p\Big)^{\frac{1}{p}}+2C\sqrt{\varepsilon p}L\eta\Big(E\sup_{t\in[0,\Upsilon]}|\partial_z \mathbf{v}^{\varepsilon}|^{2p}\Big)^{\frac{1}{p}}\\
&&+2C\sqrt{\varepsilon p}L\Big(E\big(\int^\Upsilon_0(1+|\partial_z Y^\varepsilon|^2)dt\big)^p\Big)^{\frac{1}{p}}.
\end{eqnarray*}
Choosing $\eta$ sufficiently small such that $2C\sqrt{\varepsilon p}L\eta<\frac{1}{2}$, we deduce that
\begin{eqnarray}\notag
I_1&\leq& C|\partial_z \gamma_1|^2+2C\sqrt{\varepsilon p}L+C\varepsilon(E\sup_{t\in[0,\Upsilon]}|\mathbf{v}^{\varepsilon}|^{16p}_4)^{\frac{1}{p}}\\
\label{eq-34}
&&
+(C\varepsilon+2C\sqrt{\varepsilon p}L)\Big(E\big(\int^\Upsilon_0|\partial_z \mathbf{v}^{\varepsilon}|^2dt\big)^{2p}\Big)^{\frac{1}{p}}+C\varepsilon\Big(E\big(\int^\Upsilon_0(1+\|Y^{\varepsilon}\|^2)dt\big)^p\Big)^{\frac{1}{p}},
\end{eqnarray}
It follows from Proposition \ref{lem-2} and (\ref{eq-29}) that
\begin{eqnarray}\label{eq-35}
I_1\leq C+C(p)\varepsilon+2C(p)\sqrt{\varepsilon p}L.
\end{eqnarray}
Referring to Proposition 5.3 in \cite{D-G-T-Z}, we have
\begin{eqnarray*}
&&\sup_{t\in[0,\Upsilon]}|\partial_z T^{\varepsilon}|^2+\varepsilon\int^\Upsilon_0 \|\partial_z T^{\varepsilon}\|^2dt\\
&\leq& C|\partial_z \gamma_2|^2+C\varepsilon\int^\Upsilon_0(1+|\mathbf{v}^{\varepsilon}|^8_4)|\partial_z T^{\varepsilon}|^2dt+C\varepsilon\int^\Upsilon_0(1+|T^{\varepsilon}|^8_4)\|\partial_z \mathbf{v}^{\varepsilon}\|^2dt\\
&&+C\varepsilon\int^\Upsilon_0(1+\|Y^{\varepsilon}\|^2)dt+2\sqrt{\varepsilon}\sup_{t\in[0,\Upsilon]}|\int^t_0\langle \partial_z \psi_2(\varepsilon s,Y^{\varepsilon})dW,\ \partial_z T^{\varepsilon} \rangle|.
\end{eqnarray*}
Similar to $I_1$, we obtain
\begin{eqnarray*}
&&\Big(E\big(\sup_{t\in[0,\Upsilon]}|\partial_z T^{\varepsilon}|^2+\varepsilon\int^\Upsilon_0 \|\partial_z T^{\varepsilon}\|^2dt\big)^p\Big)^{\frac{1}{p}}\\
&\leq& C|\partial_z \gamma_2|^2+C\varepsilon\left(E\Big(\int^\Upsilon_0(1+\|Y^{\varepsilon}\|^2)dt\Big)^p\right)^{\frac{1}{p}}\\
&&+C\varepsilon\Big(E(1+\sup_{t\in[0,\Upsilon]}|\mathbf{v}^{\varepsilon}|^{16p}_4)\Big)^{\frac{1}{p}}+C\varepsilon\Big(E(\int^\Upsilon_0|\partial_z T^{\varepsilon}|^2dt)^{2p}\Big)^{\frac{1}{p}}\\
&&+C\varepsilon\Big(E(1+\sup_{t\in[0,\Upsilon]}|T^{\varepsilon}|^{16p}_4)\Big)^{\frac{1}{p}}+C\varepsilon\Big(E(\int^\Upsilon_0\|\partial_z \mathbf{v}^{\varepsilon}\|^2dt)^{2p}\Big)^{\frac{1}{p}}\\
&&+2\sqrt{\varepsilon p}L\eta(E\sup_{t\in[0,\Upsilon]}|\partial_z T^{\varepsilon}|^{2p})^{\frac{1}{p}}+2C\sqrt{\varepsilon p}L\Big(E(\int^\Upsilon_0(1+\|Y^{\varepsilon}\|^2)dt)^p\Big)^{\frac{1}{p}}.
\end{eqnarray*}
Taking $\eta$ to be small enough such that $2\sqrt{\varepsilon p}L\eta<\frac{1}{2}$, we arrive at
\begin{eqnarray}\notag
I_2&\leq& C|\partial_z \gamma_2|^2+C\varepsilon\Big(E(\int^\Upsilon_0(1+\|Y^{\varepsilon}\|^2)dt)^p\Big)^{\frac{1}{p}}\\ \notag
&&+C\varepsilon\Big(E(1+\sup_{t\in[0,\Upsilon]}|\mathbf{v}^{\varepsilon}|^{16p}_4)\Big)^{\frac{1}{p}}+C\varepsilon\Big(E(\int^\Upsilon_0|\partial_z T^{\varepsilon}|^2dt)^{2p}\Big)^{\frac{1}{p}}\\
\notag
&&+C\varepsilon\Big(E(1+\sup_{t\in[0,\Upsilon]}|T^{\varepsilon}|^{16p}_4)\Big)^{\frac{1}{p}}+C\varepsilon\Big(E(\int^\Upsilon_0\|\partial_z \mathbf{v}^{\varepsilon}\|^2dt)^{2p}\Big)^{\frac{1}{p}}\\
\label{e-24}
&&+2C\sqrt{\varepsilon p}L\Big(E(\int^\Upsilon_0(1+\|Y^{\varepsilon}\|^2)dt)^p\Big)^{\frac{1}{p}}.
\end{eqnarray}
Utilizing (\ref{eq-34}), Proposition \ref{lem-2}, Proposition \ref{prp-2} and Lemma 5.1 in \cite{D-G-T-Z}, we deduce from (\ref{e-24}) that
\begin{eqnarray}\label{e-9}
I_2\leq C+C(p)\varepsilon+2C(p)\sqrt{\varepsilon p}L.
\end{eqnarray}
From (\ref{eq-35}) and (\ref{e-9}), we have
\begin{eqnarray*}
\Big(E\big(\sup_{t\in[0,\Upsilon]}|\partial_z Y^{\varepsilon}|^2+\varepsilon\int^\Upsilon_0 \|\partial_z Y^{\varepsilon}\|^2dt\big)^p\Big)^{\frac{1}{p}} \leq C(p)(C+C\varepsilon+2C\sqrt{\varepsilon p}L).
\end{eqnarray*}
Applying the same method as (\ref{e-1}) in Proposition \ref{lem-1}, we complete the proof.
\end{proof}

 \begin{flushleft}
 \textbf{Proof of Theorem \ref{thm-2}}.
\quad In view of (\ref{e-10}), (\ref{eeee-2}) can be easily deduced by Proposition \ref{lem-1}, Proposition \ref{lem-2} and Proposition \ref{lem-3}.
\end{flushleft}
$\hfill\blacksquare$

\subsection{Proof of (\ref{eq-8})} \label{sect3.3}
To prove (\ref{eq-8}), we need the following Lemma \ref{lem-5}-Lemma \ref{lem-8}.
Since $D(A)$ is dense in $V$, there exists a sequence $\{\gamma_n\}_{n\geq1}\subset D(A)$ such that
\[
\lim_{n\rightarrow \infty}\|\gamma_n-\gamma\|=0.
\]
Denote by $Z^\varepsilon_n(\cdot)=(Z^{1,\varepsilon}_n, Z^{2,\varepsilon}_n)$ the solution of (\ref{eq-7}) with the initial value $\gamma_n\in D(A)$.
\begin{lemma}\label{lem-5}
For any $n\in \mathbb{Z}^+$,
\begin{eqnarray*}
\lim_{M\rightarrow \infty}\sup_{0<\varepsilon\leq1}\varepsilon\log P\Big(\sup_{t\in[0,\Upsilon]}|AZ^{\varepsilon}_n(t)|^2>M\Big)=-\infty.
\end{eqnarray*}
\end{lemma}
\begin{proof}
 It can be proved by applying It\^{o} formula to $|AZ^{\varepsilon}_n|^2$ and using the same argument as Lemma 3.2 in \cite{X-Z}.
\end{proof}
Now, we establish the exponential convergence of $Z^\varepsilon-Z^\varepsilon_n$.
\begin{lemma}\label{lem-7}
For any $\delta>0$,
\begin{eqnarray}\label{eq-24}
\lim_{n\rightarrow \infty} \sup_{0<\varepsilon\leq 1}\varepsilon \log P\Big(\sup_{t\in [0,\Upsilon]}\|Z^\varepsilon(t)-Z^\varepsilon_n(t)\|^2> \delta\Big)=-\infty.
\end{eqnarray}
\end{lemma}
\begin{proof}
From (\ref{eq-7}), we have
 \begin{eqnarray*}
 Z^\varepsilon(t)-Z^\varepsilon_n(t)=\gamma-\gamma_n+\sqrt{\varepsilon}\int^t_0(\psi(\varepsilon s,Z^\varepsilon )-\psi(\varepsilon s,Z^\varepsilon_n))dW(s).
 \end{eqnarray*}
 Applying It\^{o} formula to $\|Z^\varepsilon(t)-Z^\varepsilon_n(t)\|^2$, we have
\begin{eqnarray*}
&&\|Z^\varepsilon(t)-Z^\varepsilon_n(t)\|^2\\
&=&\|\gamma-\gamma_n\|^2+2\sqrt{\varepsilon}\int^t_0\langle A^{\frac{1}{2}}(Z^\varepsilon-Z^\varepsilon_n),
A^{\frac{1}{2}}(\psi(\varepsilon s,Z^\varepsilon)-\psi(\varepsilon s, Z^\varepsilon_n))dW(s)\rangle\\
&&\ +\varepsilon \int^t_0\|\psi(\varepsilon s,Z^\varepsilon)-\psi(\varepsilon s, Z^\varepsilon_n)\|^2_{\mathcal{L}_2(U;V)}ds.
\end{eqnarray*}
By H\"{o}lder inequality, Hypothesis H and (\ref{eq-11}), we obtain
\begin{eqnarray*}
&&\Big(E(\sup_{0\leq s\leq t}\|Z^\varepsilon(s)-Z^\varepsilon_n(s)\|^{2p})\Big)^{\frac{2}{p}}\\
&\leq & 2\|\gamma-\gamma_n\|^4+(2L^2\varepsilon^2+8C\varepsilon p L^2)\int^t_0\Big(E(\sup_{0\leq r\leq s}\|Z^\varepsilon(r)-Z^\varepsilon_n(r)\|^{2p})\Big)^{\frac{2}{p}}ds.
\end{eqnarray*}
Utilizing Gronwall inequality, we get
\begin{eqnarray*}
\Big(E(\sup_{0\leq t\leq \Upsilon}\|Z^\varepsilon(t)-Z^\varepsilon_n(t)\|^{2p})\Big)^{\frac{2}{p}}
\leq 2\|\gamma-\gamma_n\|^4\exp\left\{2L^2\varepsilon^2\Upsilon+8C\varepsilon p L^2\Upsilon\right\}.
\end{eqnarray*}
Applying the same argument as the proof of (\ref{e-1}) in Proposition \ref{lem-1}, we complete the proof.
\end{proof}

Let $Y^\varepsilon_n(\cdot)$ be the solution of (\ref{eq-6}) with the initial value $\gamma_n$. It follows from Theorem \ref{thm-2} that
\begin{eqnarray}\label{eq-18}
\lim_{M\rightarrow \infty}\sup_{0<\varepsilon\leq1}\varepsilon\log P\Big((|Y^{\varepsilon}_n|_V(\Upsilon))^2>M\Big)=-\infty \quad \forall \  n\in \mathbb{Z}^+.
\end{eqnarray}
Then, we verify the exponential convergence of $Y^\varepsilon-Y^\varepsilon_n$.
\begin{lemma}\label{lem-6}
For any $\delta >0$,
\begin{eqnarray*}
\lim_{n\rightarrow \infty}\sup_{0<\varepsilon\leq1}\varepsilon\log P\Big(\sup_{t\in [0,\Upsilon]}\|Y^\varepsilon(t)-Y^\varepsilon_n(t)\|^2> \delta\Big)=-\infty.
\end{eqnarray*}
\end{lemma}
\begin{proof}
 For any $n$ and $K>0$, define stopping times
\begin{eqnarray*}
\tau_{\varepsilon,K}&:=&\inf\Big\{t: \varepsilon \int^t_0 |A Y^{\varepsilon}(r)|^2dr>K, \ {\rm{or}}\ \|Y^{\varepsilon}(t)\|^2>K\Big\},\\
\tau^n_{\varepsilon,K}&:=&\inf\Big\{t: \varepsilon \int^t_0 |A Y^{\varepsilon}_n(r)|^2dr>K, \ {\rm{or}} \ \|Y^{\varepsilon}_n(t)\|^2>K\Big\},\\
\tau^n_\varepsilon&:=& \tau_{\varepsilon,K}\wedge\tau^n_{\varepsilon,K}.
\end{eqnarray*}
Clearly,
\begin{eqnarray}\notag
&&P\Big(\sup_{t\in [0,\Upsilon]}\|Y^\varepsilon(t)-Y^\varepsilon_n(t)\|^2> \delta, \ (|Y^\varepsilon|_V(\Upsilon))^2 \leq K, \ (|Y^\varepsilon_n|_V(\Upsilon))^2 \leq K\Big)\\
\label{e-14}
&\leq& P\Big(\sup_{t\in [0,\Upsilon\wedge \tau^n_\varepsilon]}\|Y^\varepsilon(t)-Y^\varepsilon_n(t)\|^2> \delta\Big).
\end{eqnarray}
Let $l$ be a positive constant. Applying It\^{o} formula to
\[
{\rm{e}}^{-l\varepsilon \int^{t\wedge \tau^n_\varepsilon}_0(\|Y^\varepsilon\|^2|AY^\varepsilon|^2+\|Y^\varepsilon_n\|^2|AY^\varepsilon_n|^2)ds}
\|Y^\varepsilon(t\wedge \tau^n_\varepsilon)-Y^\varepsilon_n(t\wedge \tau^n_\varepsilon)\|^2:=\chi(t\wedge \tau^n_\varepsilon)\|Y^\varepsilon(t\wedge \tau^n_\varepsilon)-Y^\varepsilon_n(t\wedge \tau^n_\varepsilon)\|^2,
\]
 we get
\begin{eqnarray*}
&&\chi(t\wedge \tau^n_\varepsilon)\|Y^\varepsilon(t\wedge \tau^n_\varepsilon)-Y^\varepsilon_n(t\wedge \tau^n_\varepsilon)\|^2
+2\varepsilon  \int^{t\wedge \tau^n_\varepsilon}_0\chi(s)|AY^\varepsilon-AY^\varepsilon_n|^2ds\\
&=&\|\gamma-\gamma_n\|^2-l\varepsilon\int^{t\wedge \tau^n_\varepsilon}_0\chi(s)(\|Y^\varepsilon\|^2|AY^\varepsilon|^2+\|Y^\varepsilon_n\|^2|AY^\varepsilon_n|^2)\|Y^\varepsilon-Y^\varepsilon_n\|^2ds\\
&&-2\varepsilon\int^{t\wedge \tau^n_\varepsilon}_0\chi(s)\langle B(Y^\varepsilon,Y^\varepsilon )-B(Y^\varepsilon_n,Y^\varepsilon_n ),\ A(Y^\varepsilon-Y^\varepsilon_n)\rangle ds\\
&&-2\varepsilon\int^{t\wedge \tau^n_\varepsilon}_0\chi(s)\langle G(Y^\varepsilon)-G(Y^\varepsilon_n), A(Y^\varepsilon-Y^\varepsilon_n)\rangle ds\\
&&+\varepsilon\int^{t\wedge \tau^n_\varepsilon}_0\chi(s)\|\psi(\varepsilon s,Y^\varepsilon(s) )-\psi(\varepsilon s,Y^\varepsilon_n(s) )\|^2_{\mathcal{L}_2(U;V)}ds\\
&&+2\sqrt{\varepsilon}\int^{t\wedge \tau^n_\varepsilon}_0\chi(s)\langle (\psi(\varepsilon s,Y^\varepsilon(s) )-\psi(\varepsilon s,Y^\varepsilon_n(s) ))dW(s),\  A(Y^\varepsilon-Y^\varepsilon_n) \rangle.
\end{eqnarray*}
Recall (5.9) in \cite{D-G-T}, it gives
\begin{eqnarray}\notag
&&|\langle B(Y^\varepsilon,Y^\varepsilon )-B(Y^\varepsilon_n,Y^\varepsilon_n ), A(Y^\varepsilon-Y^\varepsilon_n)\rangle|\\
\label{e-11}
&\leq& \frac{1}{4}|A(Y^\varepsilon-Y^\varepsilon_n)|^2
+C(\|Y^\varepsilon\|^2|AY^\varepsilon|^2+\|Y^\varepsilon_n\|^2|AY^\varepsilon_n|^2)\|Y^\varepsilon-Y^\varepsilon_n\|^2.
\end{eqnarray}
By H\"{o}lder inequality and the Young's inequality, we have
\begin{eqnarray}\label{e-12}
|\langle G(Y^\varepsilon)-G(Y^\varepsilon_n), A(Y^\varepsilon-Y^\varepsilon_n)\rangle|
\leq \frac{1}{2}|AY^\varepsilon-AY^\varepsilon_n|^2+C\|Y^\varepsilon-Y^\varepsilon_n\|^2.
\end{eqnarray}
Thus, we deduce from Hypothesis H, (\ref{e-11}) and (\ref{e-12}) that
\begin{eqnarray}\notag
&&\chi(t\wedge \tau^n_\varepsilon)\|Y^\varepsilon(t\wedge \tau^n_\varepsilon)-Y^\varepsilon_n(t\wedge \tau^n_\varepsilon)\|^2
+2\varepsilon  \int^{t\wedge \tau^n_\varepsilon}_0\chi(s)|AY^\varepsilon-AY^\varepsilon_n|^2ds\\ \notag
&\leq&\|\gamma-\gamma_n\|^2-l\varepsilon\int^{t\wedge \tau^n_\varepsilon}_0\chi(s)(\|Y^\varepsilon\|^2|AY^\varepsilon|^2+\|Y^\varepsilon_n\|^2|AY^\varepsilon_n|^2)\|Y^\varepsilon-Y^\varepsilon_n\|^2ds\\ \notag
&&+\frac{\varepsilon}{2}\int^{t\wedge \tau^n_\varepsilon}_0\chi(s)|AY^\varepsilon-AY^\varepsilon_n|^2ds+2\varepsilon L\int^{t\wedge \tau^n_\varepsilon}_0\chi(s)\|Y^\varepsilon-Y^\varepsilon_n\|^2ds\\
\notag
&&+2\varepsilon C\int^{t\wedge \tau^n_\varepsilon}_0\chi(s)(\|Y^\varepsilon\|^2|AY^\varepsilon|^2+\|Y^\varepsilon_n\|^2|AY^\varepsilon_n|^2)\|Y^\varepsilon-Y^\varepsilon_n\|^2ds\\
\notag
&&+\varepsilon\int^{t\wedge \tau^n_\varepsilon}_0\chi(s)|AY^\varepsilon-AY^\varepsilon_n|^2 ds
+ \varepsilon C \int^{t\wedge \tau^n_\varepsilon}_0\chi(s)\|Y^\varepsilon -Y^\varepsilon_n\|^2ds\\
\label{e-13}
&&+2\sqrt{\varepsilon}\int^{t\wedge \tau^n_\varepsilon}_0\chi(s)\langle (\psi(\varepsilon s,Y^\varepsilon(s) )-\psi(\varepsilon s,Y^\varepsilon_n(s) ))dW(s),\  A(Y^\varepsilon-Y^\varepsilon_n) \rangle.
\end{eqnarray}
Choosing $l>2C$ in (\ref{e-13}), we obtain
\begin{eqnarray*}
&&\chi(t\wedge \tau^n_\varepsilon)\|Y^\varepsilon(t\wedge \tau^n_\varepsilon)-Y^\varepsilon_n(t\wedge \tau^n_\varepsilon)\|^2\\
&\leq&\|\gamma-\gamma_n\|^2+2\varepsilon L\int^{t\wedge \tau^n_\varepsilon}_0\chi(s)\|Y^\varepsilon-Y^\varepsilon_n\|^2ds+C\varepsilon \int^{t\wedge \tau^n_\varepsilon}_0\chi(s)\|Y^\varepsilon-Y^\varepsilon_n\|^2ds\\
&&+2\sqrt{\varepsilon}\int^{t\wedge \tau^n_\varepsilon}_0\chi(s)\langle (\psi(\varepsilon s,Y^\varepsilon(s) )-\psi(\varepsilon s,Y^\varepsilon_n(s) ))dW(s),\  A(Y^\varepsilon-Y^\varepsilon_n) \rangle.
\end{eqnarray*}
Utilizing H\"{o}lder inequality, we get
\begin{eqnarray*}
&&\Big(E\big(\sup_{0\leq s\leq t\wedge \tau^n_\varepsilon}(\chi(s)\|Y^\varepsilon(s)-Y^\varepsilon_n(s)\|^2)\big)^p\Big)^{\frac{2}{p}}\\
&\leq&2\|\gamma-\gamma_n\|^4+4\varepsilon^2 L^2\int^{t}_0\Big(E\big(\sup_{0\leq r\leq s\wedge \tau^n_\varepsilon}(\chi(r)\|Y^\varepsilon-Y^\varepsilon_n\|^2)^p\big)\Big)^{\frac{2}{p}}ds\\
&&+C\varepsilon^2 \int^{t}_0\Big(E\big(\sup_{0\leq r\leq s\wedge \tau^n_\varepsilon}(\chi(r)\|Y^\varepsilon-Y^\varepsilon_n\|^2)^p\big)\Big)^{\frac{2}{p}}ds+ 16C\varepsilon p L^2 \int^t_0\Big(E\big(\sup_{0\leq r\leq s\wedge \tau^n_\varepsilon}(\chi^2(r)\|Y^\varepsilon-Y^\varepsilon_n\|^4\big)^{\frac{p}{2}}\Big)^{\frac{2}{p}}ds\\
&\leq&2\|\gamma-\gamma_n\|^4+4\varepsilon^2 L^2\int^{t}_0\Big(E\big(\sup_{0\leq r\leq s\wedge \tau^n_\varepsilon}(\chi(r)\|Y^\varepsilon-Y^\varepsilon_n\|^2)^p\big)\Big)^{\frac{2}{p}}ds\\
&&+C\varepsilon^2 \int^{t}_0\Big(E\big(\sup_{0\leq r\leq s\wedge \tau^n_\varepsilon}(\chi(r)\|Y^\varepsilon-Y^\varepsilon_n\|^2)^p\big)\Big)^{\frac{2}{p}}ds+ 16C\varepsilon p L^2 \int^{t}_0\Big(E\big(\sup_{0\leq r\leq s\wedge \tau^n_\varepsilon}(\chi(r)\|Y^\varepsilon-Y^\varepsilon_n\|^2)^p\big)\Big)^{\frac{2}{p}}ds.
\end{eqnarray*}
Applying Gronwall inequality, we have
\begin{eqnarray*}
\Big(E\big(\sup_{0\leq t\leq \Upsilon\wedge \tau^n_\varepsilon}(\chi(t)\|Y^\varepsilon(t)-Y^\varepsilon_n(t)\|^2)\big)^p\Big)^{\frac{2}{p}}
\leq 2\|\gamma-\gamma_n\|^4 {\rm{e}}^{4\varepsilon^2 L^2\Upsilon+16C\varepsilon p L^2\Upsilon+C\varepsilon^2 \Upsilon}.
\end{eqnarray*}
Thus, we get
\begin{eqnarray}\notag
&&\Big(E(\sup_{0\leq t\leq \Upsilon\wedge \tau^n_\varepsilon}\|Y^\varepsilon(t)-Y^\varepsilon_n(t)\|^2\big)^p\Big)^{\frac{2}{p}}\\ \notag
&\leq&  (E[\sup_{0\leq t\leq \Upsilon\wedge \tau^n_\varepsilon}(\chi(t)\|Y^\varepsilon(t)-Y^\varepsilon_n(t)\|^2)^p\chi^{-p}( \Upsilon\wedge \tau^n_\varepsilon)])^{\frac{2}{p}}\\ \notag
&\leq&  {\rm{e}}^{8lK^2}\Big(E\sup_{0\leq t\leq \Upsilon\wedge \tau^n_\varepsilon}\big(\chi(t)\|Y^\varepsilon(t)-Y^\varepsilon_n(t)\|^2\big)^p\Big)^{\frac{2}{p}}\\
\label{e-23}
&\leq& 2{\rm{e}}^{8lK^2}\|\gamma-\gamma_n\|^4{\rm{e}}^{4\varepsilon^2 L^2\Upsilon+16C\varepsilon p L^2\Upsilon+C\varepsilon^2 \Upsilon}.
\end{eqnarray}
Taking $p=\frac{2}{\varepsilon}$ in (\ref{e-23}),
\begin{eqnarray}\notag
&&\sup_{0<\varepsilon \leq 1}\varepsilon \log P\Big(\sup_{t\in [0,\Upsilon\wedge \tau^n_\varepsilon]}\|Y^\varepsilon(t)-Y^\varepsilon_n(t)\|^2> \delta\Big)\\ \notag
&\leq& \sup_{0<\varepsilon \leq 1}\varepsilon \log \frac{E(\sup_{0\leq t\leq \Upsilon\wedge \tau^n_\varepsilon}\|Y^\varepsilon(t)-Y^\varepsilon_n(t)\|^2)^p}{\delta^p}\\
\label{eq-19}
&\leq& 8lK^2+4 L^2\Upsilon+32C L^2\Upsilon+C\Upsilon-2\log \delta +\log (2\|\gamma-\gamma_n\|^4).
\end{eqnarray}
From  (\ref{e-14}) and (\ref{eq-19}), letting $n\rightarrow \infty$ on both sides of (\ref{eq-19}), we have
\begin{eqnarray}\label{e-15}
\sup_{0<\varepsilon \leq 1}\varepsilon \log P\Big(\sup_{t\in [0,\Upsilon]}\|Y^\varepsilon(t)-Y^\varepsilon_n(t)\|^2> \delta, \ (|Y^\varepsilon|_V(\Upsilon))^2 \leq K, \ (|Y^\varepsilon_n|_V(\Upsilon))^2 \leq K\Big)\rightarrow -\infty \quad \forall K.
\end{eqnarray}
Recall Theorem \ref{thm-2}, for any $R>0$, there exists a constant $M_1$ such that
 \begin{eqnarray}\label{eq-21}
\sup_{0<\varepsilon \leq 1}\varepsilon \log P((|Y^\varepsilon|_V(\Upsilon))^2>M_1)\leq -R.
 \end{eqnarray}
Moreover, for such $R$, by (\ref{eq-18}), there exists a constant $M_2$ such that
\begin{eqnarray}\label{eq-22}
\sup_{0<\varepsilon \leq 1}\varepsilon \log P((|Y^\varepsilon_n|_V(\Upsilon))^2>M_2)\leq -R \quad \forall n,
\end{eqnarray}
and by (\ref{e-15}), there exists a positive integer $N$, such that for any $n\geq N$,
\begin{eqnarray}\label{e-17}
\sup_{0<\varepsilon \leq 1}\varepsilon \log P\Big(\sup_{t\in [0,\Upsilon]}\|Y^\varepsilon(t)-Y^\varepsilon_n(t)\|^2> \delta, \ (|Y^\varepsilon|_V(\Upsilon))^2 \leq K, \ (|Y^\varepsilon_n|_V(\Upsilon))^2 \leq K\Big)\leq -R \quad \forall K.
\end{eqnarray}
Let $n>N$ in (\ref{eq-22}) and $K=M:=M_1\vee M_2$ in (\ref{e-17}), we deduce that for any $n\geq N$,
\begin{eqnarray*}
&&\sup_{0<\varepsilon \leq 1}\varepsilon \log P\Big(\sup_{t\in [0,\Upsilon]}\|Y^\varepsilon(t)-Y^\varepsilon_n(t)\|^2> \delta\Big)\\
&\leq&  \sup_{0<\varepsilon \leq 1}\varepsilon \log P\Big(\sup_{t\in [0,\Upsilon]}\|Y^\varepsilon(t)-Y^\varepsilon_n(t)\|^2> \delta, \ (|Y^\varepsilon|_V(\Upsilon))^2 \leq M, \ (|Y^\varepsilon_n|_V(\Upsilon))^2 \leq M\Big)\\
&&+\sup_{0<\varepsilon \leq 1}\varepsilon \log P((|Y^\varepsilon|_V(\Upsilon))^2>M)
+\sup_{0<\varepsilon \leq 1}\varepsilon \log P((|Y^\varepsilon_n|_V(\Upsilon))^2>M)\\
&\leq&-3R.
\end{eqnarray*}
Since $R$ is arbitrary, we complete the proof.
\end{proof}

Finally, we study the exponential convergence of $Y^\varepsilon_n-Z^\varepsilon_n$.
\begin{lemma}\label{lem-8}
For any $\delta>0$ and every positive integer $n$,
\begin{eqnarray}\label{eq-25}
\lim_{\varepsilon\rightarrow 0} \varepsilon \log P\Big(\sup_{t\in [0,\Upsilon]}\|Y^\varepsilon_n(t)-Z^\varepsilon_n(t)\|^2> \delta\Big)=-\infty.
\end{eqnarray}
\end{lemma}
\begin{proof}
For any $K>0$ and $n$, define
\begin{eqnarray*}
\tau^{1,n}_{\varepsilon, K}&:=&\inf\Big\{t\geq 0: \|Y^\varepsilon_n(t)\|^2 >K, \ {\rm{or}}\ \varepsilon\int^t_0 |AY^\varepsilon_n(s)|^2ds >K\Big\},\\
\tau^{2,n}_{\varepsilon, K}&:=& \inf\Big\{t\geq 0: |AZ^\varepsilon_n(t)|^2 >K\Big\},\\
\tau^{n}_{\varepsilon, K}&:=&\tau^{1,n}_{\varepsilon, K}\wedge\tau^{2,n}_{\varepsilon, K},
\end{eqnarray*}
then, we have
\begin{eqnarray}\notag
&&P\Big(\sup_{t\in [0,\Upsilon]}\|Y^\varepsilon_n(t)-Z^\varepsilon_n(t)\|^2> \delta,\ (|Y^\varepsilon_n|_V(\Upsilon))^2\leq K,\ \sup_{t\in[0,\Upsilon]}|AZ^\varepsilon_n(t)|^2 \leq K\Big)\\
\label{e-16}
&\leq& P\Big(\sup_{t\in [0,\Upsilon\wedge \tau^{n}_{\varepsilon, K}]}\|Y^\varepsilon_n(t)-Z^\varepsilon_n(t)\|^2> \delta).
\end{eqnarray}
Applying It\^{o} formula to $\|Y^\varepsilon_n(t\wedge \tau^{n}_{\varepsilon, K})-Z^\varepsilon_n(t\wedge \tau^{n}_{\varepsilon, K})\|^2$, we obtain
\begin{eqnarray*}
&&\|Y^\varepsilon_n(t\wedge \tau^{n}_{\varepsilon, K})-Z^\varepsilon_n(t\wedge \tau^{n}_{\varepsilon, K})\|^2+2\varepsilon \int^{t\wedge \tau^{n}_{\varepsilon, K}}_0|A(Y^\varepsilon_n-Z^\varepsilon_n)|^2ds\\
&=&-2\varepsilon \int^{t\wedge \tau^{n}_{\varepsilon, K}}_0\langle AZ^\varepsilon_n, A(Y^\varepsilon_n-Z^\varepsilon_n)\rangle ds\\
&&-2\varepsilon \int^{t\wedge \tau^{n}_{\varepsilon, K}}_0\langle B(Y^\varepsilon_n,Y^\varepsilon_n), A(Y^\varepsilon_n-Z^\varepsilon_n)\rangle ds\\
&&-2\varepsilon \int^{t\wedge \tau^{n}_{\varepsilon, K}}_0\langle G(Y^\varepsilon_n), A(Y^\varepsilon_n-Z^\varepsilon_n)\rangle ds\\
&&+\varepsilon \int^{t\wedge \tau^{n}_{\varepsilon, K}}_0\|\psi(\varepsilon s,Y^\varepsilon_n)-\psi(\varepsilon s,Z^\varepsilon_n)\|^2_{\mathcal{L}_2(U;V)}ds\\
&&+2\sqrt{\varepsilon}\int^{t\wedge \tau^{n}_{\varepsilon, K}}_0\langle (\psi(\varepsilon s,Y^\varepsilon_n)-\psi(\varepsilon s,Z^\varepsilon_n))dW(s), \ A(Y^\varepsilon_n-Z^\varepsilon_n)\rangle.
\end{eqnarray*}
Note that
\[
\langle B(Y^\varepsilon_n,Y^\varepsilon_n), A(Y^\varepsilon_n-Z^\varepsilon_n)\rangle
=\langle B(Y^\varepsilon_n-Z^\varepsilon_n,Y^\varepsilon_n), A(Y^\varepsilon_n-Z^\varepsilon_n)\rangle
+\langle B(Z^\varepsilon_n,Y^\varepsilon_n), A(Y^\varepsilon_n-Z^\varepsilon_n)\rangle.
\]
By the Cauchy-Schwarz inequality and the Young's inequality, we have
\begin{eqnarray}\notag
&&|\langle B(Y^\varepsilon_n-Z^\varepsilon_n,Y^\varepsilon_n), \ A(Y^\varepsilon_n-Z^\varepsilon_n)\rangle|\\ \notag
&\leq& \frac{1}{8} |A(Y^\varepsilon_n-Z^\varepsilon_n)|^2+C\Big(|AY^\varepsilon_n|^2+|\frac{\partial T^\varepsilon_n}{\partial z}|^2|\nabla\frac{\partial T^\varepsilon_n}{\partial z}|^2\Big)\| \mathbf{v}^{\varepsilon}_n-Z^{1,\varepsilon}_n\|^2\\
\label{e-18}
&&+C|\frac{\partial \mathbf{v}^{\varepsilon}_n}{\partial z}|^2|\nabla\frac{\partial \mathbf{v}^{\varepsilon}_n}{\partial z}|^2,
\end{eqnarray}
and
\begin{eqnarray}\notag
&&|\langle B(Z^\varepsilon_n,Y^\varepsilon_n), A(Y^\varepsilon_n-Z^\varepsilon_n)\rangle|\\ \notag
&\leq& \frac{1}{8} |A(Y^\varepsilon_n-Z^\varepsilon_n)|^2+C|\nabla \mathbf{v}^{\varepsilon}_n ||\Delta \mathbf{v}^{\varepsilon}_n|\|Z^{1,\varepsilon}_n\|^2+C|\nabla Z^{1,\varepsilon}_n|^2|\Delta Z^{1,\varepsilon}_n |^2\\
\label{e-19}
&&+C|\frac{\partial Y^\varepsilon_n}{\partial z}|^2|\nabla\frac{\partial Y^\varepsilon_n}{\partial z}|^2+C\|T^\varepsilon_n\|^2|\Delta \mathbf{v}^{\varepsilon}_n|^2.
\end{eqnarray}
Based on (\ref{e-18}) and (\ref{e-19}), we deduce that
\begin{eqnarray*}
&&|\langle B(Y^\varepsilon_n,Y^\varepsilon_n), A(Y^\varepsilon_n-Z^\varepsilon_n)\rangle|\\
&\leq& \frac{1}{4}| A(Y^\varepsilon_n-Z^\varepsilon_n)|^2+C(|A Y^\varepsilon_n|^2+|\frac{\partial T^\varepsilon_n }{\partial z}|^2|\nabla \frac{\partial T^\varepsilon_n }{\partial z}|^2)\|\mathbf{v}^{\varepsilon}_n-Z^{1, \varepsilon}_n\|^2\\
&&+C|\frac{\partial Y^\varepsilon_n }{\partial z}|^2|\nabla \frac{\partial Y^\varepsilon_n }{\partial z}|^2
+C(\|Z^\varepsilon_n\|^2+\|T^\varepsilon_n\|^2)|\Delta \mathbf{v}^{\varepsilon}_n|^2+C|\nabla Z^{1,\varepsilon}_n |^2|\Delta Z^{1,\varepsilon}_n|^2.
\end{eqnarray*}
Moreover, by H\"{o}lder inequality and the Young's inequality, we get
\begin{eqnarray*}
|\langle G(Y^\varepsilon_n), A(Y^\varepsilon_n-Z^\varepsilon_n)\rangle|
&\leq& \frac{1}{4}|A(Y^\varepsilon_n-Z^\varepsilon_n)|^2+C\|Y^\varepsilon_n\|^2.
\end{eqnarray*}
Based on the above estimates and using Hypothesis H, we obtain
\begin{eqnarray*}
&&\|Y^\varepsilon_n(t\wedge \tau^{n}_{\varepsilon, K})-Z^\varepsilon_n(t\wedge \tau^{n}_{\varepsilon, K})\|^2+2\varepsilon \int^{t\wedge \tau^{n}_{\varepsilon, K}}_0|A(Y^\varepsilon_n-Z^\varepsilon_n)|^2ds\\
&\leq&2\varepsilon \int^{t\wedge \tau^{n}_{\varepsilon, K}}_0| AZ^\varepsilon_n| |A(Y^\varepsilon_n-Z^\varepsilon_n)|ds
+\varepsilon \int^{t\wedge \tau^{n}_{\varepsilon, K}}_0|A(Y^\varepsilon_n-Z^\varepsilon_n)|^2ds\\
&&+2\varepsilon C \int^{t\wedge \tau^{n}_{\varepsilon, K}}_0(|A Y^\varepsilon_n|^2+|\frac{\partial T^\varepsilon_n }{\partial z}|^2|\nabla \frac{\partial T^\varepsilon_n }{\partial z}|^2)\|Y^\varepsilon_n-Z^{ \varepsilon}_n\|^2ds\\
&&+2\varepsilon C \int^{t\wedge \tau^{n}_{\varepsilon, K}}_0 \Big[(\|Z^\varepsilon_n\|^2+\|T^\varepsilon_n\|^2)|\Delta \mathbf{v}^{\varepsilon}_n|^2+|\nabla Z^{1,\varepsilon}_n |^2|\Delta Z^{1,\varepsilon}_n|^2\Big]ds\\
&&+2\varepsilon C \int^{t\wedge \tau^{n}_{\varepsilon, K}}_0 |\frac{\partial Y^\varepsilon_n}{\partial z}|^2|\nabla\frac{\partial Y^\varepsilon_n}{\partial z}|^2ds\\
&&+\varepsilon C\int^{t\wedge \tau^{n}_{\varepsilon, K}}_0\|Y^\varepsilon_n\|^2 ds
+\varepsilon L^2 \int^{t\wedge \tau^{n}_{\varepsilon, K}}_0\|Y^\varepsilon_n - Z^\varepsilon_n\|^2ds\\
&&+2\sqrt{\varepsilon}\int^{t\wedge \tau^{n}_{\varepsilon, K}}_0\langle (\psi(\varepsilon s,Y^\varepsilon_n)-\psi(\varepsilon s,Z^\varepsilon_n))dW(s), \ A(Y^\varepsilon_n-Z^\varepsilon_n)\rangle.
\end{eqnarray*}
By the Cauchy-Schwarz inequality, it gives
\begin{eqnarray*}
&&\|Y^\varepsilon_n(t\wedge \tau^{n}_{\varepsilon, K})-Z^\varepsilon_n(t\wedge \tau^{n}_{\varepsilon, K})\|^2\\
&\leq&4\varepsilon \int^{t\wedge \tau^{n}_{\varepsilon, K}}_0| AZ^\varepsilon_n|^2ds
+2\varepsilon C \int^{t\wedge \tau^{n}_{\varepsilon, K}}_0(|A Y^\varepsilon_n|^2+|\frac{\partial T^\varepsilon_n }{\partial z}|^2|\nabla \frac{\partial T^\varepsilon_n }{\partial z}|^2)\|Y^\varepsilon_n-Z^{ \varepsilon}_n\|^2ds\\
&&+2\varepsilon C \int^{t\wedge \tau^{n}_{\varepsilon, K}}_0 [(\|Z^\varepsilon_n\|^2+\|T^\varepsilon_n\|^2)|\Delta \mathbf{v}^{\varepsilon}_n|^2+|\nabla Z^{1,\varepsilon}_n |^2|\Delta Z^{1,\varepsilon}_n|^2+|\frac{\partial Y^\varepsilon_n}{\partial z}|^2|\nabla\frac{\partial Y^\varepsilon_n}{\partial z}|^2]ds\\
&&+\varepsilon C\int^{t\wedge \tau^{n}_{\varepsilon, K}}_0\|Y^\varepsilon_n\|^2 ds
+\varepsilon L^2 \int^{t\wedge \tau^{n}_{\varepsilon, K}}_0\|Y^\varepsilon_n - Z^\varepsilon_n\|^2ds\\
&&+2\sqrt{\varepsilon}\int^{t\wedge \tau^{n}_{\varepsilon, K}}_0\langle (\psi(\varepsilon s,Y^\varepsilon_n)-\psi(\varepsilon s,Z^\varepsilon_n))dW(s), \ A(Y^\varepsilon_n-Z^\varepsilon_n)\rangle .
\end{eqnarray*}
Let
\[
M_t:=2\sqrt{\varepsilon}\int^{t\wedge \tau^{n}_{\varepsilon, K}}_0\langle (\psi(\varepsilon s,Y^\varepsilon_n)-\psi(\varepsilon s,Z^\varepsilon_n))dW(s),\ A(Y^\varepsilon_n-Z^\varepsilon_n)\rangle,\quad M^*_\Upsilon:=\sup_{t\in[0,\Upsilon\wedge \tau^{n}_{\varepsilon, K}]}|M_t|.
\]
Applying Gronwall inequality, we have
\begin{eqnarray*}
&&\|Y^\varepsilon_n(t\wedge \tau^{n}_{\varepsilon, K})-Z^\varepsilon_n(t\wedge \tau^{n}_{\varepsilon, K})\|^2\\
&\leq&\Big(4\varepsilon \int^{t\wedge \tau^{n}_{\varepsilon, K}}_0| AZ^\varepsilon_n|^2ds
+2\varepsilon C \int^{t\wedge \tau^{n}_{\varepsilon, K}}_0 [(\|Z^\varepsilon_n\|^2+\|T^\varepsilon_n\|^2)|\Delta \mathbf{v}^{\varepsilon}_n|^2
+|\nabla Z^{1,\varepsilon}_n |^2|\Delta Z^{1,\varepsilon}_n|^2+|\frac{\partial Y^\varepsilon_n}{\partial z}|^2|\nabla\frac{\partial Y^\varepsilon_n}{\partial z}|^2]ds\\
&&+\varepsilon C\int^{t\wedge \tau^{n}_{\varepsilon, K}}_0\|Y^\varepsilon_n\|^2 ds+M^*_\Upsilon\Big)
 \cdot \exp\Big\{2\varepsilon C \int^{t\wedge \tau^{n}_{\varepsilon, K}}_0(|A Y^\varepsilon_n|^2+|\frac{\partial T^\varepsilon_n }{\partial z}|^2|\nabla \frac{\partial T^\varepsilon_n }{\partial z}|^2)ds+\varepsilon L^2t\Big\}.
\end{eqnarray*}
By (\ref{eq-11}),
\begin{eqnarray}\notag
(E|M^*_\Upsilon|^{p})^{\frac{1}{p}}
&\leq& \frac{1}{2}(E\sup_{t\in [0,\Upsilon]}\|Y^\varepsilon_n-Z^\varepsilon_n\|^{2p})^{\frac{1}{p}}+2C\sqrt{\varepsilon p}L\Big(E\big(\int^\Upsilon_0\|Y^\varepsilon_n-Z^\varepsilon_n\|^2ds\big)^p\Big)^{\frac{1}{p}}.
\end{eqnarray}
Utilizing H\"{o}lder inequality, it yields
\begin{eqnarray*}
&&\Big(E\sup_{s\in[0,\Upsilon\wedge \tau^{n}_{\varepsilon, K}]}\|Y^\varepsilon_n(s)-Z^\varepsilon_n(s)\|^{2p}\Big)^{\frac{2}{p}}\\
&\leq&{\rm{e}}^{4\varepsilon CK+4\varepsilon CK^2+2\varepsilon L\Upsilon}\Big[16\varepsilon^2\Big(E( \int^{\Upsilon\wedge \tau^{n}_{\varepsilon, K}}_0| AZ^\varepsilon_n|^2ds)^p\Big)^{\frac{2}{p}}\\
&&
+4\varepsilon^2 C \left(E\big(\int^{\Upsilon\wedge \tau^{n}_{\varepsilon, K}}_0 \Big((\|Z^\varepsilon_n\|^2+\|T^\varepsilon_n\|^2)|\Delta \mathbf{v}^{\varepsilon}_n|^2
+|\nabla Z^{1,\varepsilon}_n |^2|\Delta Z^{1,\varepsilon}_n|^2+|\frac{\partial Y^\varepsilon_n}{\partial z}|^2|\nabla\frac{\partial Y^\varepsilon_n}{\partial z}|^2\Big)ds\big)^p\right)^{\frac{2}{p}}\\
&&+\varepsilon^2C\Big(E\big( \int^{\Upsilon\wedge \tau^{n}_{\varepsilon, K}}_0\|Y^\varepsilon_n\|^2 ds\big)^p\Big)^{\frac{2}{p}}+8C\varepsilon pL^2\int^\Upsilon_0\Big(E\sup_{r\in[0,s\wedge \tau^{n}_{\varepsilon, K}]}\|Y^\varepsilon_n(r)-Z^\varepsilon_n(r)\|^{2p}\Big)^{\frac{2}{p}}ds\Big]\\
&\leq&{\rm{e}}^{4\varepsilon CK+4\varepsilon CK^2+2\varepsilon L\Upsilon}\Big[64\varepsilon^2 C K^4+\varepsilon^2CK^2
+8C\varepsilon pL^2\int^\Upsilon_0(E\sup_{r\in[0,s\wedge \tau^{n}_{\varepsilon, K}]}\|Y^\varepsilon_n(r)-Z^\varepsilon_n(r)\|^{2p})^{\frac{2}{p}}ds\Big].
\end{eqnarray*}
Applying Gronwall inequality again, we get
\begin{eqnarray}\notag
&&(E\sup_{s\in[0,\Upsilon\wedge \tau^{n}_{\varepsilon, K}]}\|Y^\varepsilon_n(s)-Z^\varepsilon_n(s)\|^{2p})^{\frac{2}{p}} \\
\label{e-20}
&\leq&{\rm{e}}^{4\varepsilon CK+4\varepsilon CK^2+2\varepsilon L\Upsilon}(64\varepsilon^2 C K^4+\varepsilon^2CK^2)\cdot \exp\Big\{8C\varepsilon p\Upsilon L^4 {\rm{e}}^{4\varepsilon CK+4\varepsilon CK^2+2\varepsilon L\Upsilon}\Big\}.
\end{eqnarray}
Choosing $p=\frac{2}{\varepsilon}$ in (\ref{e-20}), we have for any $n$ and $K>0$,
\begin{eqnarray}\notag
&&\varepsilon \log P(\sup_{t\in[0,\Upsilon\wedge \tau^{n}_{\varepsilon, K}]}\|Y^\varepsilon_n(t)-Z^\varepsilon_n(t)\|^2>\delta)\\ \notag
&\leq& \log(E\sup_{s\in[0,\Upsilon\wedge \tau^{n}_{\varepsilon, K}]}\|Y^\varepsilon_n(s)-Z^\varepsilon_n(s)\|^{2p})^{\frac{2}{p}} -\log \delta^2\\ \notag
&\leq& 4\varepsilon CK+4\varepsilon CK^2+2\varepsilon L\Upsilon+\log(64\varepsilon^2 C K^4+\varepsilon^2CK^2)\\ \notag
&&+8C\varepsilon p\Upsilon L^2 {\rm{e}}^{4\varepsilon CK+4\varepsilon CK^2+2\varepsilon L\Upsilon}-\log \delta^2\\
\label{e-21}
&\rightarrow& -\infty \quad as\ \varepsilon \rightarrow 0.
\end{eqnarray}
Recall (\ref{eq-18}), for any $R>0$, there exists a positive constant $M_1$ such that
 \begin{eqnarray}\label{eq-27}
\sup_{0<\varepsilon \leq 1}\varepsilon \log P((|Y^\varepsilon_n|_V(\Upsilon))^2>M_1)&\leq& -R \quad \forall n.
\end{eqnarray}
Moreover, for such $R>0$, by Lemma \ref{lem-5}, there exists a positive constant $M_2$ such that
\begin{eqnarray}\label{eq-28}
\sup_{0<\varepsilon \leq 1}\varepsilon \log P(\sup_{t\in [0,\Upsilon]}|AZ^\varepsilon_n|^2>M_2)&\leq& -R\quad \forall n,
\end{eqnarray}
 and by (\ref{e-16}), (\ref{e-21}), there exists $\varepsilon_0$ such that for any $0<\varepsilon\leq \varepsilon_0$,
\begin{eqnarray}\label{e-22}
\varepsilon \log P(\sup_{t\in[0,\Upsilon]}\|Y^\varepsilon_n(t)-Z^\varepsilon_n(t)\|^2>\delta,\ (|Y^\varepsilon_n|(\Upsilon))^2\leq K,\ \sup_{t\in [0,\Upsilon]}|AZ^\varepsilon_n|^2\leq K )
\leq  -R \quad \forall n, \quad \forall K.
\end{eqnarray}
Let $K$ in (\ref{e-22}) be $M:=M_1\vee M_2$. From (\ref{eq-27})-(\ref{e-22}), we deduce that for any $n$ and $0<\varepsilon\leq \varepsilon_0$,
 \begin{eqnarray*}
 &&\varepsilon \log P(\sup_{t\in[0,\Upsilon]}\|Y^\varepsilon_n(t)-Z^\varepsilon_n(t)\|^2>\delta)\\
 &\leq& \varepsilon \log P(\sup_{t\in[0,\Upsilon]}\|Y^\varepsilon_n(t)-Z^\varepsilon_n(t)\|^2>\delta,\ (|Y^\varepsilon_n|_V(\Upsilon))^2\leq M,\ \sup_{t\in [0,\Upsilon]}|AZ^\varepsilon_n|^2\leq M)\\
 &&+\varepsilon \log P( (|Y^\varepsilon_n|_V(\Upsilon))^2> M)+\varepsilon \log P(\sup_{t\in [0,\Upsilon]}|AZ^\varepsilon_n|^2> M)\\
 &\leq & -3R.
\end{eqnarray*}
Since $R$ is arbitrary, we complete the proof.
\end{proof}

Up to now, we are ready to prove (\ref{eq-8}).

\begin{flushleft}
\textbf{Proof of (\ref{eq-8}).}\quad From Lemma \ref{lem-7} and Lemma \ref{lem-6}, for any $R>0$, there exists a positive constant $N_0$ satisfying
 \begin{eqnarray*}
 P(\sup_{t\in[0,\Upsilon]}\|Z^\varepsilon(t)-Z^\varepsilon_{N_0}(t)\|^2>\delta)
\leq {\rm{e}}^{-\frac{R}{\varepsilon}} \quad \forall  \ \varepsilon \in (0,1],
\end{eqnarray*}
and
\begin{eqnarray*}
 P(\sup_{t\in[0,\Upsilon]}\|Y^\varepsilon(t)-Y^\varepsilon_{N_0}(t)\|^2>\delta)
\leq {\rm{e}}^{-\frac{R}{\varepsilon}} \quad \forall \ \varepsilon \in (0,1].
\end{eqnarray*}
Moreover, in view of Lemma \ref{lem-8}, for such $N_0$, there exists $\varepsilon_0$ such that for any $0<\varepsilon\leq \varepsilon_0$,
 \begin{eqnarray*}
 P(\sup_{t\in[0,\Upsilon]}\|Y^\varepsilon_{N_0}(t)-Z^\varepsilon_{N_0}(t)\|^2>\delta)
\leq {\rm{e}}^{-\frac{R}{\varepsilon}}.
\end{eqnarray*}
As a result of the previous inequalities and by the triangle inequality, we have for any $0<\varepsilon\leq \varepsilon_0$,
\begin{eqnarray*}
 P(\sup_{t\in[0,\Upsilon]}\|Y^\varepsilon(t)-Z^\varepsilon(t)\|^2>\delta)
\leq 3 {\rm{e}}^{-\frac{R}{\varepsilon}}.
\end{eqnarray*}
Since $R$ is arbitrary, we conclude that
\begin{eqnarray*}
 \lim_{\varepsilon\rightarrow0}\varepsilon \log P(\sup_{t\in[0,\Upsilon]}\|Y^\varepsilon(t)-Z^\varepsilon(t)\|^2>\delta)=-\infty.
\end{eqnarray*}
\end{flushleft}
$\hfill\blacksquare$

\noindent{\small {\bf  Acknowledgements}\   This work was supported by National Natural Science Foundation of China (NSFC) (No. 11431014)},
Key Laboratory of Random Complex Structures and Data Science, Academy of Mathematics and Systems Science, Chinese Academy of Sciences(No. 2008DP173182).

\def\refname{ References}

\end{document}